\documentclass[preprint,10pt]{elsarticle}

\usepackage{amsmath}
\usepackage{amssymb}
\usepackage{amsthm}

\newtheorem{thm}{Theorem}[section]
\newtheorem{cor}[thm]{Corollary}
\newtheorem{lem}[thm]{Lemma}

\newtheorem{defin}[thm]{Definition}
\newtheorem{rem}[thm]{Remark}

\numberwithin{equation}{section}

\begin{document}

\begin{frontmatter}



\title{Boundedness of multilinear Littlewood--Paley operators with convolution type kernels on products of BMO spaces}

%% \address[label1]{}[Wang,Hua]

\author{Runzhe Zhang and Hua Wang}
\address{School of Mathematics and System Science, Xinjiang University,\\
Urumqi 830046, P. R. China\\
\textbf{Dedicated to the memory of Li Xue}}
\ead{wanghua@pku.edu.cn}

\begin{abstract}
In this paper, the authors establish the existence and boundedness of multilinear Littlewood--Paley operators on products of BMO spaces, including the multilinear $g$-function, multilinear Lusin's area integral and multilinear $g^{\ast}_{\lambda}$-function. The authors prove that if the above multilinear operators are finite for a single point, then they are finite almost everywhere. Moreover, it is shown that these multilinear operators are bounded from $\mathrm{BMO}(\mathbb R^n)\times\cdots\times \mathrm{BMO}(\mathbb R^n)$ into $\mathrm{BLO}(\mathbb R^n)$ (the space of functions with bounded lower oscillation), which is a proper subspace of $\mathrm{BMO}(\mathbb R^n)$ (the space of functions with bounded mean oscillation). The corresponding estimates for multilinear Littlewood--Paley operators with non-convolution type kernels are also discussed.
\end{abstract}


\end{frontmatter}

\section{Introduction and preliminaries}

\subsection{Linear Littlewood--Paley operators}
\label{sec11}
In this paper, the sets of all real numbers and natural numbers are denoted by $\mathbb R$ and $\mathbb N$, respectively. Let $n\in \mathbb{N}$ and $\mathbb R^n$ be the $n$-dimensional Euclidean space endowed with the Lebesgue measure $dx$. The Euclidean norm of $x=(x_1,x_2,\dots,x_n)\in \mathbb R^n$ is given by
\begin{equation*}
|x|:=\bigg(\sum_{i=1}^n x_i^2\bigg)^{1/2}.
\end{equation*}
It is well known that the Littlewood--Paley theory is a very important tool in harmonic analysis, complex analysis and PDEs. Littlewood--Paley theory can be viewed as a profound generalization of the Pythagorean theorem. It originated in the 1930's and developed in the 1960's. The Littlewood--Paley function in one dimension was first introduced by Littlewood and Paley in studying the dyadic decomposition of Fourier series (see \cite{littlewood,littlewood2,littlewood3}). The Littlewood--Paley function of higher dimensions was first defined and studied by Stein (see \cite{emstein1,emstein2,stein}). Let us now recall the classical Littlewood--Paley operators on $\mathbb R^n$, which include $g$-function, Lusin's area integral and $g^{\ast}_{\lambda}$-function. Let $u(x,t):=P_t*f(x)$ be the Poisson integral of $f$, where
\begin{equation*}
P_t(x):=c_n\cdot\frac{t}{(t^2+|x|^2)^{{(n+1)}/2}}\quad \& \quad c_n=\frac{\Gamma((n+1)/2)}{\pi^{(n+1)/2}}
\end{equation*}
denotes the Poisson kernel in ${\mathbb R}^{n+1}_+$. Then the classical Littlewood--Paley $g$-function of $f$ is defined by
\begin{equation*}
g(f)(x):=\bigg(\int_0^\infty \big|\nabla u(x,t)\big|^2 t\,dt\bigg)^{1/2},
\end{equation*}
where
\begin{equation*}
\nabla=\Big(\frac{\partial}{\partial t},\frac{\partial}{\partial x_1},\dots,\frac{\partial}{\partial x_n}\Big)\quad \& \quad \big|\nabla u(x,t)\big|^2=\Big|\frac{\partial u}{\partial t}\Big|^2+\sum_{j=1}^n\Big|\frac{\partial u}{\partial x_j}\Big|^2.
\end{equation*}
The classical Lusin's area integral (also referred to as the square function) and Littlewood--Paley $g^{\ast}_{\lambda}$-function are defined, respectively, by
\begin{equation*}
S(f)(x):=\bigg(\iint_{\Gamma(x)}\big|\nabla u(y,t)\big|^2 t^{1-n}\,dydt\bigg)^{1/2}
\end{equation*}
and
\begin{equation*}
g^*_\lambda(f)(x):=\bigg(\iint_{{\mathbb R}^{n+1}_+}\bigg(\frac t{t+|x-y|}\bigg)^{\lambda n}
\big|\nabla u(y,t)\big|^2 t^{1-n}\,dydt\bigg)^{1/2}, \quad \lambda>1,
\end{equation*}
where
\begin{equation*}
\Gamma(x):=\Big\{(y,t)\in{\mathbb R}^{n+1}_+:|y-x|<t\Big\}~~\mbox{and}~~
{\mathbb R}^{n+1}_+=\Big\{(y,t)\in{\mathbb R}^{n+1}:y\in\mathbb R^n,t>0\Big\}.
\end{equation*}

\begin{itemize}
  \item By the Plancherel formula, we can easily see that the classical Littlewood--Paley operators are bounded on $L^2(\mathbb R^n)$.
  \item Let $1<p<\infty$. Stein proved that the Littlewood--Paley $g$-function can characterize $L^p$ spaces. Moreover, there exist two positive constants $C_1$ and $C_2$, independent of $f$, such that
  \begin{equation}\label{11}
  C_1\|f\|_{L^p}\leq\|g(f)\|_{L^p}\leq C_2\|f\|_{L^p},
  \end{equation}
  for every $f\in L^p(\mathbb R^n)$. The above estimate also holds for Lusin's area integral $S(f)$ and Littlewood--Paley $g^{\ast}_{\lambda}$-function $g^*_\lambda(f)$ when $\lambda>2$. For the proofs of these results, see Stein \cite{emstein1,emstein2,stein} and Fefferman \cite{fe}.
\end{itemize}
It is well known that $u(x,t)=P_t*f(x)$ satisfies Laplace's equation $\Delta u=0$ in ${\mathbb R}^{n+1}_+$,
\begin{equation*}
\Delta u(x,t)=\Delta_x u(x,t)+\frac{\partial^2u(x,t)}{\partial t^2}=0,
\end{equation*}
and has boundary values equal to $f$, in the sense that
\begin{equation*}
\lim_{t\to 0}u(x,t)=\lim_{t\to 0}P_t*f(x)=f(x)
\end{equation*}
almost everywhere. Moreover, $u(\cdot,t)\rightarrow f(\cdot)$ in $L^p(\mathbb R^n)$ if $f\in L^p(\mathbb R^n)$ with $1\leq p<\infty$. The estimates for classical Littlewood--Paley operators rely heavily on tricks from classical harmonic analysis and partial differential equations (Green's theorem, the mean value property of harmonic functions, etc).

We now consider the following more general Littlewood--Paley operators on $\mathbb R^n$. Let $\psi$ be a real-valued function on $\mathbb R^n$ satisfying the following three conditions.
\begin{itemize}
  \item[(i)] (\textbf{The vanishing condition}):
  \begin{equation}\label{psi1}
  \psi\in L^1(\mathbb R^n)\quad \mbox{and} \quad \int_{\mathbb R^n}\psi(x)\,dx=0;
  \end{equation}
  \item[(ii)] (\textbf{the size condition}): there exist two positive constants $C_1$ and $\delta$ such that
  \begin{equation}\label{psi2}
  |\psi(x)|\leq C_1\cdot\frac{1}{(1+|x|)^{n+\delta}};
  \end{equation}
  \item[(iii)] (\textbf{the smoothness condition}): there exist two positive constants $C_2$ and $\gamma$ such that
  \begin{equation}\label{psi3}
  |\psi(x+y)-\psi(x)|\leq C_2\cdot\frac{|y|^{\gamma}}{(1+|x|)^{n+\delta+\gamma}},
  \end{equation}
  whenever $2|y|\leq |x|$.
\end{itemize}
For such a function $\psi$, the generalized Littlewood--Paley $g$-function $g_{\psi}$, Lusin's area integral $S_{\psi}$ and Littlewood--Paley $g^{\ast}_{\lambda}$-function $g^*_{\lambda,\psi}$ are defined as follows:
\begin{equation*}
g_{\psi}(f)(x):=\bigg(\int_0^\infty \big|\psi_t*f(x)\big|^2\frac{dt}{t}\bigg)^{1/2},
\end{equation*}
\begin{equation*}
S_{\psi}(f)(x):=\bigg(\iint_{\Gamma(x)}\big|\psi_t*f(y)\big|^2\frac{dydt}{t^{n+1}}\bigg)^{1/2},
\end{equation*}
and
\begin{equation*}
g^*_{\lambda,\psi}(f)(x):=\bigg(\iint_{{\mathbb R}^{n+1}_+}\bigg(\frac t{t+|x-y|}\bigg)^{\lambda n}
\big|\psi_t*f(y)\big|^2\frac{dydt}{t^{n+1}}\bigg)^{1/2}, \quad \lambda>1,
\end{equation*}
where for any function $\psi$ and for any $t\in(0,\infty)$, we denote
\begin{equation*}
\psi_t(x):=\frac{1}{t^n}\psi\Big(\frac{x}{\,t\,}\Big).
\end{equation*}
\begin{itemize}
  \item Denote by $\mathcal{G}$ the collection of all functions $\psi$ satisfying \eqref{psi1}, \eqref{psi2} and \eqref{psi3}.
  By the classical theory of vector-valued singular integral operators, we can also obtain the strong-type $(p,p)$ ($1<p<\infty$) and weak-type $(1,1)$ estimates for generalized (real-variable) Littlewood--Paley operators, including the generalized Littlewood--Paley $g$-function,Lusin's area integral and Littlewood--Paley $g^{\ast}_{\lambda}$-function (these operators are sublinear and non-negative).
  \item As in \eqref{11}, it was shown that the generalized Littlewood--Paley $g$-function $g_{\psi}$ can also characterize $L^p$ spaces. For any $1<p<\infty$ and $\psi\in\mathcal{G}$, then there exist two positive constants $C_1$ and $C_2$, independent of $f$, such that
  \begin{equation}\label{12}
  C_1\|f\|_{L^p}\leq\|g_{\psi}(f)\|_{L^p}\leq C_2\|f\|_{L^p},
  \end{equation}
  for all $f\in L^p(\mathbb R^n)$. Moreover, the above estimate also holds for generalized Lusin's area integral $S_{\psi}(f)$ and Littlewood--Paley $g^{\ast}_{\lambda}$-function $g^*_{\lambda,\psi}(f)$ when $\lambda>2$. More details can be found in
  \cite[Chapter 5]{lu}, \cite[Chapter XII]{tor}, \cite[Chapter 6]{wilson2} and \cite[Theorem 1.1]{xueding}.
\end{itemize}
A few historical remarks are given as follows:
\begin{enumerate}
  \item In 1985, Wang \cite{wang} first studied the behavior of classical Littlewood--Paley $g$-function acting on $L^{\infty}(\mathbb R^n)$ and $\mathrm{BMO}(\mathbb R^n)$, and proved the following result. If $f\in \mathrm{BMO}(\mathbb R^n)$, then $g(f)$ is either
infinite everywhere or finite almost everywhere, and in the latter case, there is a positive constant $C$ depending only on the dimension $n$ such that
\begin{equation*}
\big\|g(f)\big\|_{\mathrm{BMO}}\leq C\big\|f\big\|_{\mathrm{BMO}}.
\end{equation*}
The above interesting result also holds for the classical Lusin's area integral $S(f)$ and Littlewood--Paley $g^{\ast}_{\lambda}$-function $g^{\ast}_{\lambda}(f)$, which was established by Kurtz \cite{k} in 1987.
  \item Subsequently, in 2004, Sun \cite{sun} and Yabuta \cite{ya} studied the existence and boundedness properties of generalized Littlewood--Paley operators on BMO spaces (and Campanato spaces), and proved the following result. Suppose that $\psi\in L^1(\mathbb R^n)$ satisfies \eqref{psi1}, \eqref{psi2} with $\delta=1$, and the condition
\begin{equation}\label{psi4}
\big|\nabla\psi(x)\big|\leq C_2\cdot\frac{1}{(1+|x|)^{n+2}},
\end{equation}
where $\nabla:=(\partial/{\partial x_1},\dots,\partial/{\partial x_n})$ and $C_2$ is a positive constant independent of $x=(x_1,\dots,x_n)\in\mathbb R^n$. Then the generalized Littlewood--Paley $g$-function $g_{\psi}$ is bounded on $\mathrm{BMO}(\mathbb R^n)$. More precisely, if $f\in \mathrm{BMO}(\mathbb R^n)$ and $g_{\psi}(f)(x_0)<+\infty$ for a single point $x_0\in \mathbb R^n$, then $g_{\psi}(f)$ is finite almost everywhere, and there exists a positive constant $C>0$, independent of $f$, such that
\begin{equation*}
\big\|g_{\psi}(f)\big\|_{\mathrm{BMO}}\leq C\big\|f\big\|_{\mathrm{BMO}}.
\end{equation*}
Similar results for generalized Lusin's area integral and Littlewood--Paley $g^{\ast}_{\lambda}$-function were also obtained in \cite{sun,ya}.
\item In 1990, Leckband \cite{leck} established the boundedness of the square of the Littlewood--Paley $g$-function, Lusin's area integral and Littlewood--Paley $g^{\ast}_{\lambda}$-function from $L^{\infty}(\mathbb R^n)$ into $\mathrm{BLO}(\mathbb R^n)$, which is a proper subspace of $\mathrm{BMO}(\mathbb R^n)$. More precisely, Leckband proved that if $f\in L^{\infty}(\mathbb R^n)$, then there exists a positive constant $C>0$, independent of $f$, such that
\begin{equation}\label{citeleck}
\big\|[T_{g}(f)]^2\big\|_{\mathrm{BLO}}\leq C\big\|f\big\|_{L^{\infty}}^2,
\end{equation}
where $T_{g}(f)$ denotes any one of the usual classical or generalized Littlewood--Paley functions (see \cite[Theorem 1]{leck}).
\end{enumerate}

In 2008, Meng and Yang \cite{meng} further discussed the behavior of generalized Littlewood--Paley operators on BMO spaces. Let $g_{\psi}(f)$ be the generalized Littlewood--Paley $g$-function of $f$ on $\mathbb R^n$. Meng and Yang proved that if $f\in \mathrm{BMO}(\mathbb R^n)$, then $g_{\psi}(f)$ is either infinite everywhere or finite almost everywhere, and in the latter case, $[g_{\psi}(f)]^2$ is bounded from $\mathrm{BMO}(\mathbb R^n)$ into $\mathrm{BLO}(\mathbb R^n)$(see \cite[Theorem 1.1]{meng}), which is an improvement of the result of Leckband.

\begin{thm}\label{thm12}
Suppose that $\psi\in L^1(\mathbb R^n)$ satisfies \eqref{psi1}, \eqref{psi2} with $\delta=1$ and \eqref{psi4}. If $f\in \mathrm{BMO}(\mathbb R^n)$, then $g_{\psi}(f)$ is either infinite everywhere or finite almost everywhere, and in the latter case, there exists a positive constant $C>0$, independent of $f$, such that
\begin{equation*}
\big\|[g_{\psi}(f)]^2\big\|_{\mathrm{BLO}}\leq C\big\|f\big\|_{\mathrm{BMO}}^2.
\end{equation*}
\end{thm}
Similar results for generalized Lusin's area integral and Littlewood--Paley $g^{\ast}_{\lambda}$-function were also obtained by Meng and Yang(see \cite[Theorems 1.2 and 1.3]{meng}). The corresponding estimates for Marcinkiewicz integrals can be found in \cite{hu}.

We remark that the condition \eqref{psi4} implies \eqref{psi3}, by applying the mean value theorem. Arguing as in the proof of Theorem 1.1 in \cite{meng}, we can also show that the conclusions of the above theorem still hold for the generalized Littlewood--Paley operators $g_{\psi}$, $S_{\psi}$ and $g^{\ast}_{\psi}$, under the conditions \eqref{psi1}, \eqref{psi2} and \eqref{psi3} on $\psi$.

\subsection{Multilinear Littlewood--Paley operators}
\label{sec13}
In recent years, the theory of multilinear operators in harmonic analysis has attracted much attention, see, for example, \cite{grafakos2,grafakos3,grafakos4,grafakos5,lerner,li1,li2} and the references therein.
\footnote{The multilinear (Calder\'{o}n--Zygmund) theory was originated in the works of Coifman and Meyer. Later on this theory was systematically studied by Grafakos and Torres. Multilinear Calder\'{o}n--Zygmund theory is a natural generalization of the linear case.}

Let $2\leq m\in\mathbb N$ and $(\mathbb R^n)^m=\overbrace{\mathbb R^n\times\cdots\times\mathbb R^n}^m$ be the $m$-fold product space. Suppose that each function $f_i$ is locally integrable on $\mathbb R^n$, $i=1,2,\dots,m$. Let $\vec{f}$ denote the vector function $\vec{f}:=(f_1,f_2,\dots,f_m)$. 
Recently, the theory of multilinear Littlewood--Paley operators was first introduced and studied by Xue et al. in \cite{he,shi,xueqing,xuepeng}. The class of multilinear Littlewood--Paley operators with standard convolution type kernels provides the starting point of the theory (see \cite{shi,xuepeng}). We first recall the definition of the multilinear Littlewood--Paley kernel (of convolution type).
\begin{defin}[\cite{shi,xuepeng}]\label{defin12}
Let $2\leq m\in\mathbb N$. We say that a function $\mathcal{K}(y_1,y_2,\dots,y_m)$ defined on $(\mathbb R^n)^m$ is a multilinear Littlewood--Paley kernel, if the following three conditions are satisfied.
\begin{enumerate}
  \item (\textbf{The vanishing condition}): for $i=1,2,\dots,m$,
  \begin{equation*}
  \int_{\mathbb R^n}\mathcal{K}(y_1,\dots,y_i,\dots,y_m)\,dy_i=0;
  \end{equation*}
  \item (\textbf{the size condition}): for some positive constants $C$ and $\delta$,
  \begin{equation*}
  \big|\mathcal{K}(y_1,\dots,y_i,\dots,y_m)\big|\leq C\cdot\frac{1}{(1+\sum_{j=1}^m|y_j|)^{mn+\delta}};
  \end{equation*}
  \item (\textbf{the smoothness condition}): for some positive constants $C$ and $\gamma$,
  \begin{equation*}
  \big|\mathcal{K}(y_1,\dots,y_i,\dots,y_m)-\mathcal{K}(y_1,\dots,y'_i,\dots,y_m)\big|
  \leq C\cdot\frac{|y_i-y_i'|^{\gamma}}{(1+\sum_{j=1}^m|y_j|)^{mn+\delta+\gamma}},
  \end{equation*}
  whenever $2|y_i-y_i'|\leq \max_{1\leq j\leq m}|y_j|$ for all $1\leq i\leq m$.
\end{enumerate}
\end{defin}
Now we give the definition of the multilinear Littlewood--Paley operators, including multilinear $g$-function, multilinear Lusin's area integral and multilinear Littlewood--Paley $g^{\ast}_{\lambda}$-function with convolution type kernels.

\begin{defin}[\cite{shi,xuepeng}]
For any $\vec{f}=(f_1,\dots,f_m)\in \overbrace{\mathcal{S}(\mathbb R^n)\times\cdots\times \mathcal{S}(\mathbb R^n)}^m$ and any $t>0$, we denote
\begin{equation*}
\mathcal{K}_t(y_1,y_2,\dots,y_m):=\frac{1}{t^{mn}}\mathcal{K}\Big(\frac{y_1}{t},\frac{y_2}{t},\dots,\frac{y_m}{t}\Big),
\end{equation*}
and
\begin{equation*}
\mathcal{G}_t(\vec{f})(x):=\int_{(\mathbb R^n)^m}\mathcal{K}_t(x-y_1,\dots,x-y_m)\prod_{i=1}^m f_i(y_i)\,dy_i,
\quad \mbox{for all}\;\,x\notin\bigcap_{i=1}^m \mathrm{supp}\, f_i,
\end{equation*}
for a given multilinear Littlewood--Paley kernel $\mathcal{K}$. Then the multilinear Littlewood--Paley $g$-function, multilinear Lusin's area integral and multilinear Littlewood--Paley $g^{\ast}_{\lambda}$-function with convolution type kernels are defined, respectively, by
\begin{equation*}
g(\vec{f})(x):=\bigg(\int_0^{\infty}\big|\mathcal{G}_t(\vec{f})(x)\big|^2\frac{dt}{t}\bigg)^{1/2},~~
S(\vec{f})(x):=\bigg(\iint_{\Gamma(x)}\big|\mathcal{G}_t(\vec{f})(z)\big|^2\frac{dzdt}{t^{n+1}}\bigg)^{1/2},
\end{equation*}
and
\begin{equation*}
g^{\ast}_{\lambda}(\vec{f})(x):=\bigg(\iint_{\mathbb{R}^{n+1}_{+}}\Big(\frac{t}{t+|x-z|}\Big)^{\lambda n}
\big|\mathcal{G}_t(\vec{f})(z)\big|^2\frac{dzdt}{t^{n+1}}\bigg)^{1/2},\quad \lambda>1.
\end{equation*}
Throughout this paper, we will always assume that $\mathcal{T}_{g}$ can be extended to be a bounded multilinear operator for some $1\leq q_1,q_2,\dots,q_m<\infty$, $0<q<\infty$ with $1/q=\sum_{i=1}^m 1/{q_i}$; that is,
\begin{equation*}
\mathcal{T}_{g}:L^{q_1}(\mathbb R^n)\times L^{q_2}(\mathbb R^n)\times \cdots\times L^{q_m}(\mathbb R^n)\rightarrow L^q(\mathbb R^n),
\end{equation*}
where $\mathcal{T}_{g}(\vec{f})$ denotes any one of the multilinear Littlewood--Paley functions. Here we use the standard notation $\mathcal{S}(\mathbb R^n)$ for the Schwartz space of test functions on $\mathbb R^n$.
\end{defin}

When $m=1$, this definition coincides with the one given in Section \ref{sec11}.

The multilinear Littlewood--Paley $g$-function was first defined and studied by Xue--Peng--Yabuta \cite{xuepeng} in 2015. The multilinear Littlewood--Paley $g^{\ast}_{\lambda}$-function was first defined and studied by Shi--Xue--Yabuta \cite{shi} in 2014. The multilinear Littlewood--Paley operator is a natural generalization of the linear case. Thus it is natural and interesting to study the generalizations of \eqref{12} in the multilinear setting. The strong-type and weak-type estimates of multilinear Littlewood--Paley $g$-function and Lusin's area integral were given in \cite{xuepeng} and \cite{xueqing}.
The strong-type and weak-type estimates of multilinear Littlewood--Paley $g^{\ast}_{\lambda}$-function were also obtained in \cite{shi} and \cite{xueqing}. Based on the above results, in 2015, He--Xue--Mei--Yabuta further studied the existence and boundedness of multilinear Littlewood--Paley operators on BMO spaces (and Campanato spaces), and obtained the following BMO type estimates (see \cite[Corollary 1.3 and Corollary 1.4]{he}).

\begin{thm}[\cite{he}]\label{thm18}
For any $f_1,f_2\in \mathrm{BMO}(\mathbb R^n)$, if $g(f_1,f_2)$ is finite for a single point $x_0\in \mathbb R^n$, then $g(f_1,f_2)$ is finite almost everywhere on $\mathbb R^n$, and there exists a positive constant $C>0$, independent of $f_1$ and $f_2$, such that
\begin{equation*}
\big\|g(f_1,f_2)\big\|_{\mathrm{BMO}}\leq C\big\|f_1\big\|_{\mathrm{BMO}}\big\|f_2\big\|_{\mathrm{BMO}}.
\end{equation*}
If $S(f_1,f_2)$ is finite for a single point $x_0\in \mathbb R^n$, then $S(f_1,f_2)$ is finite almost everywhere on $\mathbb R^n$, and there exists a positive constant $C>0$, independent of $f_1$ and $f_2$, such that
\begin{equation*}
\big\|S(f_1,f_2)\big\|_{\mathrm{BMO}}\leq C\big\|f_1\big\|_{\mathrm{BMO}}\big\|f_2\big\|_{\mathrm{BMO}}.
\end{equation*}
\end{thm}

\begin{thm}[\cite{he}]\label{thm19}
Let $\lambda>4$. For any $f_1,f_2\in \mathrm{BMO}(\mathbb R^n)$, if $g^{\ast}_{\lambda}(f_1,f_2)$ is finite for a single point $x_0\in \mathbb R^n$, then $g^{\ast}_{\lambda}(f_1,f_2)$ is finite almost everywhere on $\mathbb R^n$, and there exists a positive constant $C>0$, independent of $f_1$ and $f_2$, such that
\begin{equation*}
\big\|g^{\ast}_{\lambda}(f_1,f_2)\big\|_{\mathrm{BMO}}\leq C\big\|f_1\big\|_{\mathrm{BMO}}\big\|f_2\big\|_{\mathrm{BMO}}.
\end{equation*}
\end{thm}

Inspired by the previous works (Theorem \ref{thm12} in the linear case, and Theorems \ref{thm18} and \ref{thm19} in the bilinear case), it is natural to ask the question whether the conclusion in Theorem \ref{thm12} still holds in the multilinear setting. In this paper, we will give a positive answer to this question.

Let $\mathcal{T}_{g}(\vec{f})$ denote the multilinear Littlewood--Paley functions of $\vec{f}$ on $\mathbb R^n$, including the multilinear $g$-function $g(\vec{f})$, multilinear Lusin's area integral $S(\vec{f})$ and multilinear Littlewood--Paley $g^{\ast}_{\lambda}$-function $g^{\ast}_{\lambda}(\vec{f})$. It is proved that if $\vec{f}=(f_1,f_2,\dots,f_m)\in[\mathrm{BMO}(\mathbb R^n)]^{m}$, then $\mathcal{T}_{g}(\vec{f})$ is either infinite everywhere or finite almost everywhere, and in the latter case, $\big[\mathcal{T}_{g}(\vec{f})\big]^2$ is bounded from $\mathrm{BMO}(\mathbb R^n)\times\cdots\times \mathrm{BMO}(\mathbb R^n)$ into $\mathrm{BLO}(\mathbb R^n)$, which is a proper subspace of $\mathrm{BMO}(\mathbb R^n)$. Moreover, we also obtain that when $\vec{f}=(f_1,f_2,\dots,f_m)\in[L^{\infty}(\mathbb R^n)]^m$, then $\mathcal{T}_{g}(\vec{f})$ is finite everywhere, and $\big[\mathcal{T}_{g}(\vec{f})\big]^2$ is bounded from $L^{\infty}(\mathbb R^n)\times\cdots\times L^{\infty}(\mathbb R^n)$ into $\mathrm{BLO}(\mathbb R^n)$, which is an extension of Leckband's result in the multilinear setting.

\section{Definitions and notations}

\subsection{Lebesgue spaces, $\mathrm{BMO}$ and $\mathrm{BLO}$ spaces}
\label{sec12}
Recall that, for any given $p\in(0,\infty)$, the Lebesgue space $L^p(\mathbb R^n)$ is defined as the set of all integrable functions $f$ on $\mathbb R^n$ such that
\begin{equation*}
\|f\|_{L^p}:=\bigg(\int_{\mathbb R^n}|f(x)|^p\,dx\bigg)^{1/p}<+\infty,
\end{equation*}
and the weak Lebesgue space $L^{p,\infty}(\mathbb R^n)$ is defined to be the set of all Lebesgue measurable functions $f$ on $\mathbb R^n$ such that
\begin{equation*}
\|f\|_{L^{p,\infty}}:=\sup_{\lambda>0}\lambda\cdot m\big(\big\{x\in\mathbb R^n:|f(x)|>\lambda\big\}\big)^{1/p}<+\infty.
\end{equation*}
Let $L^{\infty}(\mathbb R^n)$ denote the Banach space of all essentially bounded measurable functions $f$ on $\mathbb R^n$.
The norm of $f\in L^{\infty}(\mathbb R^n)$ is given by
\begin{equation*}
\|f\|_{L^\infty}:=\underset{x\in\mathbb R^n}{\mbox{ess\,sup}}\,|f(x)|<+\infty.
\end{equation*}
For any $x_0\in\mathbb R^n$ and $r>0$, let $B(x_0,r):=\{x\in\mathbb R^n:|x-x_0|<r\}$ denote the open ball centered at $x_0$ with the radius $r$, and $B(x_0,r)^{\complement}$ denote its complement. Given $B=B(x_0,r)$ and $t>0$, we will write $tB$ for the $t$-dilate ball, which is the ball with the same center $x_0$ and with radius $tr$. For a measurable set $E\subset\mathbb R^n$, we use the notation $m(E)$ for the $n$-dimensional Lebesgue measure of the set $E$, and we use the notation $\chi_{E}$ to denote the characteristic function of the set $E$: $\chi_E(x)=1$ if $x\in E$ and $0$ if $x\notin E$.

A locally integrable function $f$ on $\mathbb R^n$ is said to be in $\mathrm{BMO}(\mathbb R^n)$, the space of bounded mean oscillation(see \cite{john}), if
\begin{equation*}
\|f\|_{\mathrm{BMO}}:=\sup_{\mathcal{B}\subset\mathbb R^n}\frac{1}{m(\mathcal{B})}
\int_{\mathcal{B}}|f(x)-f_{\mathcal{B}}|\,dx<+\infty,
\end{equation*}
where $f_{\mathcal{B}}$ denotes the mean value of $f$ over $\mathcal{B}$, i.e.,
\begin{equation*}
f_{\mathcal{B}}:=\frac{1}{m(\mathcal{B})}\int_{\mathcal{B}} f(y)\,dy
\end{equation*}
and the supremum is taken over all balls $\mathcal{B}$ in $\mathbb R^n$. Modulo constants, the space $\mathrm{BMO}(\mathbb R^n)$ is a Banach function space with respect to the norm $\|\cdot\|_{\mathrm{BMO}}$. The space of BMO functions was first introduced by John and Nirenberg in \cite{john}.

A locally integrable function $f$ on $\mathbb R^n$ is said to be in $\mathrm{BLO}(\mathbb R^n)$, the space of bounded lower oscillation(see \cite{coifman}), if there exists a constant $C>0$ such that for any ball $\mathcal{B}\subset\mathbb R^n$,
\begin{equation*}
\frac{1}{m(\mathcal{B})}\int_{\mathcal{B}}\Big[f(x)-\underset{y\in\mathcal{B}}{\mathrm{ess\,inf}}\,f(y)\Big]\,dx\leq C.
\end{equation*}
The smallest constant $C$ as above is defined to be the BLO-constant of $f$, and is denoted by $\|f\|_{\mathrm{BLO}}$. The space of BLO functions was first introduced by Coifman and Rochberg in \cite{coifman}.

\subsection{Inclusion relations between $L^{\infty}$, $\mathrm{BLO}$ and $\mathrm{BMO}$}
It can be shown that
\begin{equation*}
L^{\infty}(\mathbb R^n)\subset\mathrm{BLO}(\mathbb R^n)\subset\mathrm{BMO}(\mathbb R^n).
\end{equation*}
Moreover, the above inclusion relations are both strict, see \cite{hu,meng,ou} for some examples. It is easy to verify that
\begin{equation}\label{relation11}
\|f\|_{\mathrm{BLO}}\leq 2\|f\|_{L^\infty},
\end{equation}
and
\begin{equation}\label{bmoblo}
\|f\|_{\mathrm{BMO}}\leq 2\|f\|_{\mathrm{BLO}}.
\end{equation}
In fact, suppose that $f\in L^{\infty}(\mathbb R^n)$. For any ball $\mathcal{B}\subset\mathbb R^n$, it is easy to see that
\begin{equation*}
\begin{split}
&\frac{1}{m(\mathcal{B})}\int_{\mathcal{B}}\Big[f(x)-\underset{y\in\mathcal{B}}{\mathrm{ess\,inf}}\,f(y)\Big]\,dx\\
&\leq \frac{1}{m(\mathcal{B})}\int_{\mathcal{B}}\big[2\|f\|_{L^\infty}\big]dx=2\|f\|_{L^\infty}.
\end{split}
\end{equation*}
This proves \eqref{relation11}. On the other hand, let $f$ belong to $\mathrm{BLO}(\mathbb R^n)$. Then for any ball $\mathcal{B}\subset\mathbb R^n$,
\begin{equation*}
\begin{split}
&\frac{1}{m(\mathcal{B})}\int_{\mathcal{B}}\big|f(x)-f_{\mathcal{B}}\big|\,dx\\
&=\frac{1}{m(\mathcal{B})}\int_{\mathcal{B}}\Big|f(x)
-\underset{y\in\mathcal{B}}{\mathrm{ess\,inf}}\,f(y)+\underset{y\in\mathcal{B}}{\mathrm{ess\,inf}}\,f(y)-f_{\mathcal{B}}\Big|\,dx\\
&\leq\frac{1}{m(\mathcal{B})}\int_{\mathcal{B}}\Big[f(x)-\underset{y\in\mathcal{B}}{\mathrm{ess\,inf}}\,f(y)\Big]\,dx
+\Big|\underset{y\in\mathcal{B}}{\mathrm{ess\,inf}}\,f(y)-f_{\mathcal{B}}\Big|\\
&\leq\frac{2}{m(\mathcal{B})}\int_{\mathcal{B}}\Big[f(x)-\underset{y\in\mathcal{B}}{\mathrm{ess\,inf}}\,f(y)\Big]\,dx
\leq 2\|f\|_{\mathrm{BLO}},
\end{split}
\end{equation*}
as desired. This proves \eqref{bmoblo}.

\begin{rem}
It should be pointed out that $\|\cdot\|_{\mathrm{BLO}}$ is not a norm and $\mathrm{BLO}(\mathbb R^n)$ is not a linear space (it is a proper subspace of $\mathrm{BMO}(\mathbb R^n)$).
\end{rem}

Throughout this paper, we use $C$ to denote a positive constant, which is independent of main parameters and may be different at each occurrence. By $\mathbf{X}\lesssim\mathbf{Y}$, we mean that there exists a positive constant $C>0$ such that $\mathbf{X}\leq C\mathbf{Y}$. If $\mathbf{X}\lesssim\mathbf{Y}$ and $\mathbf{Y}\lesssim\mathbf{X}$, then we write $\mathbf{X}\approx\mathbf{Y}$ and say that $\mathbf{X}$ and $\mathbf{Y}$ are equivalent.

\section{Main results}
The main purpose of this paper is to establish the existence and boundedness of multilinear Littlewood--Paley operators with convolution type kernels on products of BMO spaces, including the multilinear $g$-function, multilinear Lusin's area integral and multilinear Littlewood--Paley $g^{\ast}_{\lambda}$-function. We will prove that if the above operators are finite for one point, then they are finite almost everywhere. Moreover, these multilinear operators are bounded from $\mathrm{BMO}(\mathbb R^n)\times\cdots\times \mathrm{BMO}(\mathbb R^n)$ into $\mathrm{BLO}(\mathbb R^n)$. These results can be viewed as an improvement of Theorems \ref{thm18} and \ref{thm19} in the bilinear case. To this aim, we start by giving the following results.
\begin{thm}[\cite{xueqing,xuepeng}]\label{gs1}
Let $2\leq m\in \mathbb{N}$, $1\leq p_1,p_2,\dots,p_m<\infty$ and $0<p<\infty$ with
\begin{equation*}
\frac{\,1\,}{p}=\frac{1}{p_1}+\frac{1}{p_2}+\cdots+\frac{1}{p_m}.
\end{equation*}
Then the following results hold:

$(i)$ If each $p_i>1$, $i=1,2,\dots,m$, then there is a constant $C>0$ independent of $\vec{f}$ such that
\begin{equation*}
\big\|g(\vec{f})\big\|_{L^p}\leq C\prod_{i=1}^m\|f_i\|_{L^{p_i}},
\quad \big\|S(\vec{f})\big\|_{L^p}\leq C\prod_{i=1}^m\|f_i\|_{L^{p_i}},
\end{equation*}
hold for all $\vec{f}=(f_1,f_2,\dots,f_m)\in L^{p_1}(\mathbb R^n)\times L^{p_2}(\mathbb R^n)\times \cdots\times L^{p_m}(\mathbb R^n)$.

$(ii)$ If at least one $p_i=1$, then there is a constant $C>0$ independent of $\vec{f}$ such that
\begin{equation*}
\big\|g(\vec{f})\big\|_{L^{p,\infty}}\leq C\prod_{i=1}^m\|f_i\|_{L^{p_i}},
\quad \big\|S(\vec{f})\big\|_{L^{p,\infty}}\leq C\prod_{i=1}^m\|f_i\|_{L^{p_i}},
\end{equation*}
hold for all $\vec{f}=(f_1,f_2,\dots,f_m)\in L^{p_1}(\mathbb R^n)\times L^{p_2}(\mathbb R^n)\times \cdots\times L^{p_m}(\mathbb R^n)$.
In particular, the multilinear operators $g$ and $S$ are bounded from $L^{1}(\mathbb R^n)\times L^{1}(\mathbb R^n)\times \cdots\times L^{1}(\mathbb R^n)$ into $L^{1/m,\infty}(\mathbb R^n)$.
\end{thm}

\begin{thm}[\cite{shi,xueqing}]\label{gs1ambda}
Suppose that $\lambda>2m$ and $0<\gamma<\min\big\{\delta,{n(\lambda-2m)}/2\big\}$.
Let $2\leq m\in \mathbb{N}$, $1\leq p_1,p_2,\dots,p_m<\infty$ and $0<p<\infty$ with
\begin{equation*}
\frac{\,1\,}{p}=\frac{1}{p_1}+\frac{1}{p_2}+\cdots+\frac{1}{p_m}.
\end{equation*}
Then the following results hold:

$(i)$ If each $p_i>1$, $i=1,2,\dots,m$, then there is a constant $C>0$ independent of $\vec{f}$ such that
\begin{equation*}
\big\|g^{\ast}_{\lambda}(\vec{f})\big\|_{L^p}\leq C\prod_{i=1}^m\|f_i\|_{L^{p_i}}.
\end{equation*}

$(ii)$ If  at least one $p_i$ equals one, then there is a constant $C>0$ independent of $\vec{f}$ such that
\begin{equation*}
\big\|g^{\ast}_{\lambda}(\vec{f})\big\|_{L^{p,\infty}}\leq C\prod_{i=1}^m\|f_i\|_{L^{p_i}}.
\end{equation*}
In particular, the multilinear operator $g^{\ast}_{\lambda}$ is bounded from $L^{1}(\mathbb R^n)\times L^{1}(\mathbb R^n)\times \cdots\times L^{1}(\mathbb R^n)$ into $L^{1/m,\infty}(\mathbb R^n)$ when $\lambda>2m$ and $0<\gamma<\min\big\{\delta,{n(\lambda-2m)}/2\big\}$.
\end{thm}

\begin{rem}
Note that if $m=1$, then the above theorem is just the classical result of Stein \cite{emstein2,stein} when it is associated with the Poisson kernel, and is the result of Xue and Ding \cite{xueding} when it is associated with more general kernel satisfying the conditions \eqref{psi1}, \eqref{psi2} and \eqref{psi3}. The weak-type $(1,1)$ estimate in \cite{emstein2,stein} is essentially the best possible in the sense that $\lambda>2$. It seems that the range of $\lambda>2m$ is the best possible adapted to the multilinear($m$-linear) theory.
\end{rem}

Let $2\leq m\in \mathbb{N}$. When $f_i\in\mathrm{BMO}(\mathbb R^n)$ for $i=1,2,\dots,m$, we simply write
\begin{equation*}
\vec{f}:=(f_1,f_2,\dots,f_m)\in [\mathrm{BMO}(\mathbb R^n)]^{m}.
\end{equation*}

The main results of this paper are stated as follows.

\begin{thm}\label{mainthm1}
For any $\vec{f}=(f_1,f_2,\dots,f_m)\in [\mathrm{BMO}(\mathbb R^n)]^{m}$ and $2\leq m\in \mathbb{N}$, then $g(\vec{f})$ is either infinite everywhere or finite almost everywhere, and in the latter case, there exists a positive constant $C$, independent of $\vec{f}$, such that
\begin{equation*}
\big\|\big[g(\vec{f})\big]^2\big\|_{\mathrm{BLO}}\leq C\prod_{i=1}^m\big\|f_i\big\|^2_{\mathrm{BMO}}.
\end{equation*}
\end{thm}

\begin{thm}\label{mainthm2}
For any $\vec{f}=(f_1,f_2,\dots,f_m)\in [\mathrm{BMO}(\mathbb R^n)]^{m}$ and $2\leq m\in \mathbb{N}$, then $S(\vec{f})$ is either infinite everywhere or finite almost everywhere, and in the latter case, there exists a positive constant $C$, independent of $\vec{f}$, such that
\begin{equation*}
\big\|\big[S(\vec{f})\big]^2\big\|_{\mathrm{BLO}}\leq C\prod_{i=1}^m\big\|f_i\big\|^2_{\mathrm{BMO}}.
\end{equation*}
\end{thm}
Note that for any given ball $\mathcal{B}$ in $\mathbb R^n$ and for any $x\in \mathcal{B}$, if
\begin{equation*}
\underset{y\in\mathcal{B}}{\mathrm{ess\,inf}}\,\big[\mathcal{F}(y)\big]<+\infty,
\end{equation*}
then
\begin{equation*}
\big[\mathcal{F}(x)\big]-\underset{y\in\mathcal{B}}{\mathrm{ess\,inf}}\,\big[\mathcal{F}(y)\big]\leq
\Big(\big[\mathcal{F}(x)\big]^2-\underset{y\in\mathcal{B}}{\mathrm{ess\,inf}}\,\big[\mathcal{F}(y)\big]^2\Big)^{1/2},
\end{equation*}
which in turn implies that
\begin{equation}\label{BLOsquare}
\big\|\mathcal{F}\big\|^2_{\mathrm{BLO}}\leq\big\|\mathcal{F}^2\big\|_{\mathrm{BLO}}.
\end{equation}

As an immediate consequence of Theorem \ref{mainthm1} and Theorem \ref{mainthm2}, we have the following results.

\begin{cor}
For any $\vec{f}=(f_1,f_2,\dots,f_m)\in [\mathrm{BMO}(\mathbb R^n)]^{m}$ and $2\leq m\in \mathbb{N}$, then $g(\vec{f})$ is either infinite everywhere or finite almost everywhere, and in the latter case, we have
\begin{equation*}
\big\|g(\vec{f})\big\|_{\mathrm{BLO}}\lesssim\prod_{i=1}^m\big\|f_i\big\|_{\mathrm{BMO}}.
\end{equation*}
\end{cor}

\begin{cor}
For any $\vec{f}=(f_1,f_2,\dots,f_m)\in [\mathrm{BMO}(\mathbb R^n)]^{m}$ and $2\leq m\in \mathbb{N}$, then $S(\vec{f})$ is either infinite everywhere or finite almost everywhere, and in the latter case, we have
\begin{equation*}
\big\|S(\vec{f})\big\|_{\mathrm{BLO}}\lesssim\prod_{i=1}^m\big\|f_i\big\|_{\mathrm{BMO}}.
\end{equation*}
\end{cor}
Here the implicit constant is independent of $\vec{f}=(f_1,f_2,\dots,f_m)$.

\section{Proofs of Theorems \ref{mainthm1} and \ref{mainthm2}}
In this section, we will give the proofs of Theorem \ref{mainthm1} and Theorem \ref{mainthm2}. We first remark that the (a.e.)existence of the bilinear Littlewood--Paley operators has been proved in \cite{he}, under the assumption of one point finiteness. The multilinear case $m>2$ can be shown in the same way. We can also obtain that for the multilinear Littlewood--Paley $g$-function $g(\vec{f})$ and multilinear Lusin's area integral $S(\vec{f})$, if both $g(\vec{f})(x_0)$ and $S(\vec{f})(x_0)$ are finite for some $x_0\in\mathbb R^n$, then $g(\vec{f})(x)$ and $S(\vec{f})(x)$ are finite almost everywhere. We will establish boundedness properties of the multilinear Littlewood--Paley operators in the product of BMO spaces. The following result about BMO functions is well known, see, for example, \cite{duoand} and \cite{grafakos}.

\begin{lem}\label{BMOp}
Let $f\in \mathrm{BMO}(\mathbb R^n)$. Then the following properties hold.
\begin{enumerate}
  \item For any $1\leq p<\infty$ and for any ball $\mathcal{B}$ in $\mathbb R^n$, we get
\begin{equation*}
\bigg(\frac{1}{m(\mathcal{B})}\int_{\mathcal{B}}\big|f(x)-f_{\mathcal{B}}\big|^p\,dx\bigg)^{1/p}
\leq C\big\|f\big\|_{\mathrm{BMO}}.
\end{equation*}
  \item For every $k\in \mathbb{N}$, we get
\begin{equation*}
\frac{1}{m(2^k\mathcal{B})}\int_{2^k\mathcal{B}}\big|f(x)-f_{\mathcal{B}}\big|\,dx\leq Ck\cdot\big\|f\big\|_{\mathrm{BMO}}.
\end{equation*}
\end{enumerate}
Here the constant $C$ is independent of $k$ and $f$.
\end{lem}

Let $\mathcal{F}$ be a real-valued nonnegative function and measurable on $\mathbb R^n$. For each fixed ball $\mathcal{B}\subset\mathbb R^n$, we also need the following estimate about the relationship between essential supremum and essential infimum.
\begin{equation}\label{essinf}
\mathcal{F}(x)-\underset{y\in\mathcal{B}}{\mathrm{ess\,inf}}\,\mathcal{F}(y)
\leq \underset{y\in\mathcal{B}}{\mathrm{ess\,sup}}\big|\mathcal{F}(x)-\mathcal{F}(y)\big|.
\end{equation}

\begin{proof}[Proof of Theorem $\ref{mainthm1}$]
Let $\vec{f}=(f_1,f_2,\dots,f_m)\in [\mathrm{BMO}(\mathbb R^n)]^{m}$. By the definition of $\mathrm{BLO}(\mathbb R^n)$, it suffices to show that for any given ball $\mathcal{B}=B(x_0,r)\subset\mathbb R^n$ with center $x_0\in \mathbb R^n$ and radius $r\in(0,\infty)$, the following inequality holds:
\begin{equation}\label{mainesti1}
\frac{1}{m(\mathcal{B})}\int_{\mathcal{B}}
\Big[\big[g(\vec{f})(x)\big]^2-\underset{y\in\mathcal{B}}{\mathrm{ess\,inf}}\,\big[g(\vec{f})(y)\big]^2\Big]\,dx
\lesssim\prod_{i=1}^m\big\|f_i\big\|^2_{\mathrm{BMO}}.
\end{equation}
First of all, we decompose the integral defining $g(\vec{f})$ into two parts.
\begin{equation*}
\begin{split}
\big[g(\vec{f})(x)\big]^2&=\bigg(\int_0^{\infty}\big|\mathcal{G}_t(\vec{f})(x)\big|^2\frac{dt}{t}\bigg)\\
&=\int_0^{r}\big|\mathcal{G}_t(\vec{f})(x)\big|^2\frac{dt}{t}+\int_{r}^{\infty}\big|\mathcal{G}_t(\vec{f})(x)\big|^2\frac{dt}{t}\\
&:=\big[g_0(\vec{f})(x)\big]^2+\big[g_{\infty}(\vec{f})(x)\big]^2.
\end{split}
\end{equation*}
Consequently, in view of \eqref{essinf}, we can deduce that
\begin{equation*}
\begin{split}
&\frac{1}{m(\mathcal{B})}\int_{\mathcal{B}}
\Big[\big[g(\vec{f})(x)\big]^2-\underset{y\in\mathcal{B}}{\mathrm{ess\,inf}}\,\big[g(\vec{f})(y)\big]^2\Big]\,dx\\
&=\frac{1}{m(\mathcal{B})}\int_{\mathcal{B}}\Big[\big[g_0(\vec{f})(x)\big]^2+\big[g_{\infty}(\vec{f})(x)\big]^2
-\underset{y\in\mathcal{B}}{\mathrm{ess\,inf}}\,\big[g(\vec{f})(y)\big]^2\Big]\,dx\\
&\leq\frac{1}{m(\mathcal{B})}\int_{\mathcal{B}}\Big[\big[g_0(\vec{f})(x)\big]^2+\big[g_{\infty}(\vec{f})(x)\big]^2
-\underset{y\in\mathcal{B}}{\mathrm{ess\,inf}}\,\big[g_{\infty}(\vec{f})(y)\big]^2\Big]\,dx\\
&\leq\frac{1}{m(\mathcal{B})}\int_{\mathcal{B}}\big[g_0(\vec{f})(x)\big]^2dx
+\frac{1}{m(\mathcal{B})}\int_{\mathcal{B}}
\underset{y\in\mathcal{B}}{\mathrm{ess\,sup}}\Big|\big[g_{\infty}(\vec{f})(x)\big]^2-\big[g_{\infty}(\vec{f})(y)\big]^2\Big|\,dx\\
&:=I_0+I_{\infty}.
\end{split}
\end{equation*}
Let us first estimate the term $I_0$. For any $1\leq i\leq m$, we decompose the function $f_i$ as
\begin{equation*}
f_i=(f_i)_{2\mathcal{B}}+[f_i-(f_i)_{2\mathcal{B}}]\cdot\chi_{2\mathcal{B}}
+[f_i-(f_i)_{2\mathcal{B}}]\cdot\chi_{(2\mathcal{B})^{\complement}}:=f_i^1+f_i^2+f_i^3,
\end{equation*}
where $2\mathcal{B}=B(x_0,2r)$, $(2\mathcal{B})^{\complement}=\mathbb R^n\setminus(2\mathcal{B})$ and $\chi_{E}$ denotes the
characteristic function of the set $E$. Then we write
\begin{equation}\label{sum123}
\begin{split}
\prod_{i=1}^m f_i(y_i)&=\prod_{i=1}^m\big[f_i^1(y_i)+f_i^2(y_i)+f_i^3(y_i)\big]\\
&=\sum_{\alpha_1,\dots,\alpha_m\in\{1,2,3\}}f^{\alpha_1}_1(y_1)\cdots f^{\alpha_m}_m(y_m),
\end{split}
\end{equation}
and hence
\begin{equation*}
\begin{split}
&\mathcal{G}_t(\vec{f})(x)=\mathcal{G}_t(f_1,\dots,f_m)(x)\\
&=\sum_{\alpha_1,\dots,\alpha_m\in\{1,2,3\}}\int_{(\mathbb R^n)^m}\mathcal{K}_t(x-y_1,\dots,x-y_m)
f^{\alpha_1}_1(y_1)\cdots f^{\alpha_m}_m(y_m)\,dy_1\cdots dy_m\\
&=\sum_{\alpha_1,\dots,\alpha_m\in\{1,2,3\}}\mathcal{G}_t(f^{\alpha_1}_1,\dots,f^{\alpha_m}_m)(x).
\end{split}
\end{equation*}
Observe that if $\alpha_{j}=1$ for some $1\leq j\leq m$, then $\mathcal{G}_t(f^{\alpha_1}_1,\dots,f^{\alpha_m}_m)(x)=0$ for any $x\in \mathcal{B}$, by the vanishing condition of the kernel $\mathcal{K}$.
Thus
\begin{equation*}
\begin{split}
I_0&=\frac{1}{m(\mathcal{B})}\int_{\mathcal{B}}\big[g_0(f_1,\dots,f_m)(x)\big]^2dx\\
&=\frac{1}{m(\mathcal{B})}\int_{\mathcal{B}}\big[g_0(f^2_1,\dots,f^2_m)(x)\big]^2dx+
\sum_{\alpha_1,\dots,\alpha_m\in \Xi}
\frac{1}{m(\mathcal{B})}\int_{\mathcal{B}}\big[g_0(f^{\alpha_1}_1,\dots,f^{\alpha_m}_m)(x)\big]^2dx\\
&:=I^{2,\dots,2}_0+\sum_{(\alpha_1,\dots,\alpha_m)\in\Xi}I^{\alpha_1,\dots,\alpha_m}_0,
\end{split}
\end{equation*}
where
\begin{equation*}
\Xi:=\Big\{(\alpha_1,\dots,\alpha_m):\alpha_j\in\{2,3\},\mbox{there is at least one}~ \alpha_j\neq2,1\leq j\leq m\Big\};
\end{equation*}
that is, each term of $\sum$ contains at least one $\alpha_j\neq2$. According to Theorem \ref{gs1}, we know that the $m$-linear operator $g$ is bounded from $L^{2m}(\mathbb R^n)\times\cdots\times L^{2m}(\mathbb R^n)$ into $L^2(\mathbb R^n)$. This fact, together with part (1) of Lemma \ref{BMOp}, implies that
\begin{equation*}
\begin{split}
I^{2,\dots,2}_0&\leq\frac{1}{m(\mathcal{B})}\big\|g(f^2_1,\dots,f^2_m)\big\|_{L^2}^2
\leq\frac{C}{m(\mathcal{B})}\bigg[\prod_{i=1}^m\big\|f^2_i\big\|_{L^{2m}}\bigg]^2\\
&=\frac{C}{m(\mathcal{B})}\bigg[\prod_{i=1}^m\bigg(\int_{2\mathcal{B}}
\big|f_i(y_i)-(f_i)_{2\mathcal{B}}\big|^{2m}dy_i\bigg)^{\frac{1}{2m}}\bigg]^2\\
&\lesssim\frac{1}{m(\mathcal{B})}\bigg[\prod_{i=1}^m\big\|f_i\big\|_{\mathrm{BMO}}m(2\mathcal{B})^{\frac{1}{2m}}\bigg]^2
\lesssim\prod_{i=1}^m\big\|f_i\big\|^2_{\mathrm{BMO}},
\end{split}
\end{equation*}
as desired. For the other terms, we consider the case when exactly $\ell$ of the $\alpha_i$ are $3$ for some $1\leq\ell<m$. Without
loss of generality, we may assume that
\begin{equation*}
\alpha_1=\cdots=\alpha_{\ell}=3\qquad \&\qquad \alpha_{\ell+1}=\cdots=\alpha_m=2.
\end{equation*}
The remaining terms can be done from the arguments below by permuting the indices (by symmetry of the roles of $\alpha_1,\alpha_2,\dots,\alpha_m$).
A simple geometric observation shows that
\begin{equation*}
(\mathbb R^n\setminus 2\mathcal{B})^{\ell}=
\overbrace{\big(\mathbb R^n\setminus 2\mathcal{B}\big)\times\cdots\times\big(\mathbb R^n\setminus 2\mathcal{B}\big)}^{\ell}
\subset(\mathbb R^n)^{\ell}\setminus (2\mathcal{B})^{\ell},
\end{equation*}
and
\begin{equation}\label{equaw1}
(\mathbb R^n)^{\ell}\setminus (2\mathcal{B})^{\ell}=\bigcup_{j=1}^\infty
(2^{j+1}\mathcal{B})^{\ell}\setminus(2^j\mathcal{B})^{\ell},
\end{equation}
where we have used the notation $E^{\ell}=\overbrace{E\times\cdots\times E}^{\ell}$ for a measurable set $E$ and a positive integer $\ell$ with $1\leq \ell<m$ (see Figure 1 in \cite{hanwang}).
By the size condition of the kernel $\mathcal{K}$, we have
\begin{equation*}
\begin{split}
&\big|\mathcal{G}_t\big(f^{\alpha_1}_1,\dots,f^{\alpha_m}_m\big)(x)\big|
=\big|\mathcal{G}_t\big(f^{3}_1,\dots,f^{3}_{\ell},f^2_{\ell+1},\dots,f^{2}_m\big)(x)\big|\\
&=\bigg|\int_{(\mathbb R^n\setminus 2\mathcal{B})^{\ell}}\int_{(2\mathcal{B})^{m-\ell}}\mathcal{K}_t(x-y_1,\dots,x-y_m)
f^{3}_1(y_1)\cdots f^{3}_{\ell}(y_{\ell})\cdot f^2_{\ell+1}(y_{\ell+1})\cdots f^2_m(y_m)\,dy_1\cdots dy_m\bigg|\\
&\lesssim\int_{(\mathbb R^n\setminus 2\mathcal{B})^{\ell}}
\Big|\big[f_1(y_1)-(f_1)_{2\mathcal{B}}\big]\cdots\big[f_{\ell}(y_{\ell})-(f_{\ell})_{2\mathcal{B}}\big]\Big|\,dy_1\cdots dy_{\ell}\\
&\times\int_{(2\mathcal{B})^{m-\ell}}\frac{t^{\delta}}{(t+\sum_{j=1}^m|x-y_j|)^{mn+\delta}}
\Big|\big[f_{\ell+1}(y_{\ell+1})-(f_{\ell+1})_{2\mathcal{B}}\big]\cdots\big[f_m(y_m)-(f_m)_{2\mathcal{B}}\big]\Big|\,dy_{\ell+1}\cdots dy_m\\
&\leq t^{\delta}\int_{(\mathbb R^n)^{\ell}\setminus(2\mathcal{B})^{\ell}}\frac{1}{(t+\sum_{k=1}^{\ell}|x-y_k|)^{mn+\delta}}
\Big|\big[f_1(y_1)-(f_1)_{2\mathcal{B}}\big]\cdots\big[f_{\ell}(y_{\ell})-(f_{\ell})_{2\mathcal{B}}\big]\Big|\,dy_1\cdots dy_{\ell}\\
&\times\prod_{k=\ell+1}^m\int_{2\mathcal{B}}\big|f_k(y_k)-(f_k)_{2\mathcal{B}}\big|\,dy_k.\\
\end{split}
\end{equation*}
It is easy to see that when $x\in \mathcal{B}=B(x_0,r)$ and $y\in 2^{j+1}\mathcal{B}\setminus 2^j\mathcal{B}$ with $j\geq1$,
\begin{equation}\label{factw1}
|x-y|\approx |x_0-y|.
\end{equation}
This fact, together with the equation \eqref{equaw1}, gives us that
\begin{equation*}
\begin{split}
&\big|\mathcal{G}_t\big(f^{\alpha_1}_1,\dots,f^{\alpha_m}_m\big)(x)\big|\\
&\leq t^{\delta}\sum_{j=1}^{\infty}\int_{(2^{j+1}\mathcal{B})^{\ell}\setminus(2^j \mathcal{B})^{\ell}}\frac{1}{(\sum_{k=1}^{\ell}|x-y_k|)^{mn+\delta}}
\Big|\big[f_1(y_1)-(f_1)_{2\mathcal{B}}\big]\cdots\big[f_{\ell}(y_{\ell})-(f_{\ell})_{2\mathcal{B}}\big]\Big|\,dy_1\cdots dy_{\ell}\\
&\times\prod_{k=\ell+1}^m\int_{2\mathcal{B}}\big|f_k(y_k)-(f_k)_{2\mathcal{B}}\big|\,dy_k\\
\end{split}
\end{equation*}
\begin{equation*}
\begin{split}
&\lesssim t^{\delta}\sum_{j=1}^{\infty}\bigg\{\prod_{k=1}^{\ell}
\int_{2^{j+1}\mathcal{B}\setminus 2^j\mathcal{B}}
\frac{1}{(|x_0-y_k|)^{\frac{mn+\delta}{\ell}}}\big|f_k(y_k)-(f_k)_{2\mathcal{B}}\big|\,dy_k\bigg\}
\times\prod_{k=\ell+1}^m\Big[\big\|f_k\big\|_{\mathrm{BMO}}\cdot m\big(2\mathcal{B}\big)\Big]\\
&\lesssim t^{\delta}\sum_{j=1}^{\infty}\bigg\{\prod_{k=1}^{\ell}\frac{1}{m(2^j\mathcal{B})^{\frac{mn+\delta}{\ell n}}}
\int_{2^{j+1}\mathcal{B}}\big|f_k(y_k)-(f_k)_{2\mathcal{B}}\big|\,dy_k\bigg\}
\times\prod_{k=\ell+1}^m\Big[\big\|f_k\big\|_{\mathrm{BMO}}\cdot m\big(2\mathcal{B}\big)\Big].
\end{split}
\end{equation*}
Furthermore, by using part (2) of Lemma \ref{BMOp}, we can deduce that for any $x\in\mathcal{B}$,
\begin{equation*}
\begin{split}
\big|\mathcal{G}_t\big(f^{\alpha_1}_1,\dots,f^{\alpha_m}_m\big)(x)\big|
&\lesssim t^{\delta}\sum_{j=1}^{\infty}
\bigg\{\frac{1}{m(2^j\mathcal{B})^{\frac{mn+\delta}{n}}}\prod_{k=1}^{\ell}
\Big[\big\|f_k\big\|_{\mathrm{BMO}}\cdot j m\big(2^{j+1}\mathcal{B}\big)\Big]\bigg\}\\
&\times\prod_{k=\ell+1}^m\Big[\big\|f_k\big\|_{\mathrm{BMO}}\cdot m\big(2\mathcal{B}\big)\Big]\\
&\lesssim t^{\delta}\sum_{j=1}^{\infty}\bigg\{\frac{1}{m(2^j\mathcal{B})^{\frac{mn+\delta}{n}}}\cdot j^{\ell}\bigg\}
\times\prod_{k=1}^m\Big[\big\|f_k\big\|_{\mathrm{BMO}}\cdot m\big(2^j\mathcal{B}\big)\Big]\\
&=t^{\delta}\sum_{j=1}^{\infty}\bigg\{\frac{1}{m(2^j\mathcal{B})^{\frac{\delta}{n}}}\cdot j^{\ell}\bigg\}
\times\prod_{k=1}^m\big\|f_k\big\|_{\mathrm{BMO}}.
\end{split}
\end{equation*}
Therefore,
\begin{equation*}
\begin{split}
I^{\alpha_1,\dots,\alpha_m}_0
&=\frac{1}{m(\mathcal{B})}\int_{\mathcal{B}}\big[g_0(f^{\alpha_1}_1,\dots,f^{\alpha_m}_m)(x)\big]^2dx\\
&\lesssim\frac{1}{m(\mathcal{B})}\int_{\mathcal{B}}\bigg(\int_0^{r}t^{2\delta-1}dt\bigg)
\bigg\{\sum_{j=1}^{\infty}\frac{j^{\ell}}{m(2^j\mathcal{B})^{\frac{\delta}{n}}}\bigg\}^2
\times\prod_{k=1}^m\big\|f_k\big\|^2_{\mathrm{BMO}}\\
&\lesssim r^{2\delta}\bigg\{\sum_{j=1}^{\infty}\frac{j^{\ell}}{(2^j)^{\delta}r^{\delta}}\bigg\}^2
\times\prod_{k=1}^m\big\|f_k\big\|^2_{\mathrm{BMO}}\\
&\lesssim\prod_{k=1}^m\big\|f_k\big\|^2_{\mathrm{BMO}}.
\end{split}
\end{equation*}
Let us now deal with the case when $\alpha_1=\cdots=\alpha_m=3$. In this case, we also have
\begin{equation*}
(\mathbb R^n\setminus 2\mathcal{B})^{m}=
\overbrace{\big(\mathbb R^n\setminus 2\mathcal{B}\big)\times\cdots\times\big(\mathbb R^n\setminus 2\mathcal{B}\big)}^{m}
\subset(\mathbb R^n)^{m}\setminus (2\mathcal{B})^{m},
\end{equation*}
and
\begin{equation}\label{equaw2}
(\mathbb R^n)^{m}\setminus (2\mathcal{B})^{m}=\bigcup_{j=1}^\infty
(2^{j+1}\mathcal{B})^{m}\setminus(2^j\mathcal{B})^{m},
\end{equation}
where we have used the notation $E^{m}=\overbrace{E\times\cdots\times E}^{m}$ for a measurable set $E$ and $m\in \mathbb{N}$ (see Figure 1 in \cite{hanwang}). By the size condition of the kernel $\mathcal{K}$, we can see that
\begin{equation*}
\begin{split}
&\big|\mathcal{G}_t\big(f^{\alpha_1}_1,\dots,f^{\alpha_m}_m\big)(x)\big|
=\big|\mathcal{G}_t\big(f^{3}_1,\dots,f^{3}_{m}\big)(x)\big|\\
&=\bigg|\int_{(\mathbb R^n\setminus 2\mathcal{B})^{m}}\mathcal{K}_t(x-y_1,\dots,x-y_m)
f^{3}_1(y_1)\cdots f^{3}_{m}(y_{m})\,dy_1\cdots dy_m\bigg|\\
&\lesssim\int_{(\mathbb R^n)^{m}\setminus (2\mathcal{B})^{m}}\frac{t^{\delta}}{(t+\sum_{k=1}^{m}|x-y_k|)^{mn+\delta}}
\Big|\big[f_1(y_1)-(f_1)_{2\mathcal{B}}\big]\cdots\big[f_{m}(y_{m})-(f_{m})_{2\mathcal{B}}\big]\Big|\,dy_1\cdots dy_{m}.\\
\end{split}
\end{equation*}
Hence, by \eqref{equaw2} and \eqref{factw1}, we get
\begin{equation*}
\begin{split}
\big|\mathcal{G}_t\big(f^{3}_1,\dots,f^{3}_m\big)(x)\big|
&\lesssim\sum_{j=1}^{\infty}\int_{(2^{j+1}\mathcal{B})^{m}\setminus(2^j \mathcal{B})^{m}}
\frac{t^{\delta}}{(t+\sum_{k=1}^m|x-y_k|)^{mn+\delta}}\\
&\times\Big|\big[f_1(y_1)-(f_1)_{2\mathcal{B}}\big]\cdots\big[f_{m}(y_{m})-(f_{m})_{2\mathcal{B}}\big]\Big|\,dy_1\cdots dy_{m}\\
&\lesssim t^{\delta}\sum_{j=1}^{\infty}\bigg\{\prod_{k=1}^{m}
\int_{2^{j+1}\mathcal{B}\setminus 2^j\mathcal{B}}
\frac{1}{(|x_0-y_k|)^{\frac{mn+\delta}{m}}}\big|f_k(y_k)-(f_k)_{2\mathcal{B}}\big|\,dy_k\bigg\}\\
&\lesssim t^{\delta}\sum_{j=1}^{\infty}\bigg\{\prod_{k=1}^{m}\frac{1}{m(2^j\mathcal{B})^{\frac{mn+\delta}{m n}}}
\int_{2^{j+1}\mathcal{B}}\big|f_k(y_k)-(f_k)_{2\mathcal{B}}\big|\,dy_k\bigg\}.
\end{split}
\end{equation*}
Furthermore, by using part (2) of Lemma \ref{BMOp}, we can deduce that for any $x\in\mathcal{B}$,
\begin{equation*}
\begin{split}
\big|\mathcal{G}_t\big(f^{3}_1,\dots,f^{3}_m\big)(x)\big|&\lesssim t^{\delta}\sum_{j=1}^{\infty}
\bigg\{\prod_{k=1}^{m}\frac{1}{m(2^j\mathcal{B})^{\frac{mn+\delta}{mn}}}
\Big[\big\|f_k\big\|_{\mathrm{BMO}}\cdot j\big|2^{j+1}\mathcal{B}\big|\Big]\bigg\}\\
&\lesssim t^{\delta}\sum_{j=1}^{\infty}
\bigg\{\prod_{k=1}^{m}\frac{1}{m(2^j\mathcal{B})^{\frac{\delta}{mn}}}
\Big[\big\|f_k\big\|_{\mathrm{BMO}}\cdot j\Big]\bigg\}\\
&=t^{\delta}\sum_{j=1}^{\infty}\bigg\{\frac{1}{m(2^j\mathcal{B})^{\frac{\delta}{n}}}\cdot j^{m}\bigg\}
\times\prod_{k=1}^m\big\|f_k\big\|_{\mathrm{BMO}}.
\end{split}
\end{equation*}
Therefore,
\begin{equation*}
\begin{split}
I^{3,\dots,3}_0
&=\frac{1}{m(\mathcal{B})}\int_{\mathcal{B}}\big[g_0(f^{3}_1,\dots,f^{3}_m)(x)\big]^2dx\\
&\lesssim\frac{1}{m(\mathcal{B})}\int_{\mathcal{B}}\bigg(\int_0^{r}t^{2\delta-1}dt\bigg)
\bigg\{\sum_{j=1}^{\infty}\frac{j^{m}}{m(2^j\mathcal{B})^{\frac{\delta}{n}}}\bigg\}^2
\times\prod_{k=1}^m\big\|f_k\big\|^2_{\mathrm{BMO}}\\
&\lesssim r^{2\delta}\bigg\{\sum_{j=1}^{\infty}\frac{j^{m}}{(2^j)^{\delta}r^{\delta}}\bigg\}^2
\times\prod_{k=1}^m\big\|f_k\big\|^2_{\mathrm{BMO}}\\
&\lesssim\prod_{k=1}^m\big\|f_k\big\|^2_{\mathrm{BMO}}.
\end{split}
\end{equation*}
Summing up the above estimates, we conclude that
\begin{equation*}
I_0\lesssim\prod_{i=1}^m\big\|f_i\big\|^2_{\mathrm{BMO}}.
\end{equation*}
In order to estimate the other term $I_{\infty}$, we first claim that for any $2^{k}r<t\leq 2^{k+1}r$ with $k\in \mathbb{N}\cup\{0\}$, and for any $x\in \mathcal{B}=B(x_0,r)$,
\begin{equation}\label{keyestiw1}
\big|\mathcal{G}_t(\vec{f})(x)\big|\lesssim\prod_{i=1}^m\big\|f_i\big\|_{\mathrm{BMO}}.
\end{equation}
In fact, by the size condition of the kernel $\mathcal{K}$, we know that for any $t>0$, $x,y_i\in \mathbb R^n$, $i=1,2,\dots,m$,
\begin{equation}\label{A83}
\big|\mathcal{K}_t(x-y_1,\dots,x-y_m)\big|\lesssim\frac{1}{t^{mn}}.
\end{equation}
Consequently, by the vanishing condition of the kernel $\mathcal{K}$, we have
\begin{equation}\label{firstsecond}
\begin{split}
\big|\mathcal{G}_t(\vec{f})(x)\big|&=\bigg|\int_{(\mathbb R^n)^m}\mathcal{K}_t(x-y_1,\dots,x-y_m)
\bigg(\prod_{i=1}^m\big[f_i(y_i)-(f_i)_{2^{k+1}\mathcal{B}}\big]\bigg)\,dy_1\cdots dy_m\bigg|\\
&\lesssim\int_{(2^{k+1}\mathcal{B})^m}\frac{1}{t^{mn}}\prod_{i=1}^m\big|f_i(y_i)-(f_i)_{2^{k+1}\mathcal{B}}\big|\,dy_i\\
&+\int_{(\mathbb R^n)^m\setminus(2^{k+1}\mathcal{B})^m}
\frac{t^{\delta}}{(t+\sum_{i=1}^m|x-y_i|)^{mn+\delta}}
\prod_{i=1}^m\big|f_i(y_i)-(f_i)_{2^{k+1}\mathcal{B}}\big|\,dy_i.
\end{split}
\end{equation}
Since $2^{k}r<t\leq 2^{k+1}r$, the first term in \eqref{firstsecond} is bounded by
\begin{equation*}
\begin{split}
&\prod_{i=1}^m\frac{1}{(2^{k}r)^n}\int_{2^{k+1}\mathcal{B}}\big|f_i(y_i)-(f_i)_{2^{k+1}\mathcal{B}}\big|\,dy_i\\
&\lesssim\prod_{i=1}^m\frac{1}{m(2^{k+1}\mathcal{B})}\int_{2^{k+1}\mathcal{B}}\big|f_i(y_i)-(f_i)_{2^{k+1}\mathcal{B}}\big|\,dy_i\\
&\leq\prod_{i=1}^m\big\|f_i\big\|_{\mathrm{BMO}},
\end{split}
\end{equation*}
for any $k\in \mathbb{N}\cup\{0\}$. The second term in \eqref{firstsecond} is dominated by
\begin{equation*}
\begin{split}
&\sum_{j=k+1}^{\infty}\big(2^{k+1}r\big)^{\delta}\int_{(2^{j+1}\mathcal{B})^{m}\setminus(2^j \mathcal{B})^{m}}
\frac{1}{(\sum_{i=1}^m|x-y_i|)^{mn+\delta}}\prod_{i=1}^m\big|f_i(y_i)-(f_i)_{2^{k+1}\mathcal{B}}\big|\,dy_i\\
&\lesssim\sum_{j=k+1}^{\infty}\big(2^{k+1}r\big)^{\delta}
\bigg\{\prod_{i=1}^m\frac{1}{m(2^{j+1}\mathcal{B})^{\frac{mn+\delta}{mn}}}
\int_{2^{j+1} \mathcal{B}}\big|f_i(y_i)-(f_i)_{2^{k+1}\mathcal{B}}\big|\,dy_i\bigg\},
\end{split}
\end{equation*}
where in the last step we have used the fact that when $x\in \mathcal{B}$ and $(y_1,\dots,y_m)\in (2^{j+1}\mathcal{B})^{m}\setminus(2^j \mathcal{B})^{m}$ with $j\geq k+1$,
\begin{equation*}
\sum_{i=1}^m|x-y_i|\geq\max_{1\leq k\leq m}|x-y_k|\geq 2^{j}r\cong m(2^j \mathcal{B})^{1/n}.
\end{equation*}
Moreover, in view of part (2) of Lemma \ref{BMOp}, the above expression is further dominated by
\begin{equation*}
\begin{split}
&\sum_{j=k+1}^{\infty}\frac{(2^{k+1}r)^{\delta}}{(2^{j+1}r)^{\delta}}
\bigg\{\prod_{i=1}^m\frac{1}{m(2^{j+1}\mathcal{B})}
\int_{2^{j+1} \mathcal{B}}\big|f_i(y_i)-(f_i)_{2^{k+1}\mathcal{B}}\big|\,dy_i\bigg\}\\
&\lesssim\sum_{j=k+1}^{\infty}\frac{1}{(2^{j-k})^{\delta}}\bigg\{\prod_{i=1}^m(j-k)\cdot\big\|f_i\big\|_{\mathrm{BMO}}\bigg\}\\
&=\sum_{j=1}^{\infty}\frac{j^m}{2^{j\delta}}\cdot\prod_{i=1}^m\big\|f_i\big\|_{\mathrm{BMO}}
\lesssim\prod_{i=1}^m\big\|f_i\big\|_{\mathrm{BMO}}.
\end{split}
\end{equation*}
Combining the above estimates for both terms in \eqref{firstsecond} yields the desired result \eqref{keyestiw1}. Hence, by the triangle inequality and \eqref{keyestiw1}, we obtain that for any $x,y\in \mathcal{B}=B(x_0,r)$,
\begin{equation*}
\begin{split}
&\Big|\big[g_{\infty}(\vec{f})(x)\big]^2-\big[g_{\infty}(\vec{f})(y)\big]^2\Big|
=\bigg|\int_{r}^{\infty}\big|\mathcal{G}_t(\vec{f})(x)\big|^2-\big|\mathcal{G}_t(\vec{f})(y)\big|^2\frac{dt}{t}\bigg|\\
&\leq\int_{r}^{\infty}\Big[\big|\mathcal{G}_t(\vec{f})(x)\big|+\big|\mathcal{G}_t(\vec{f})(y)\big|\Big]
\cdot\Big|\mathcal{G}_t(\vec{f})(x)-\mathcal{G}_t(\vec{f})(y)\Big|\frac{dt}{t}\\
&\lesssim\prod_{i=1}^m\big\|f_i\big\|_{\mathrm{BMO}}
\times\int_{r}^{\infty}\Big|\mathcal{G}_t(\vec{f})(x)-\mathcal{G}_t(\vec{f})(y)\Big|\frac{dt}{t}.
\end{split}
\end{equation*}
On the other hand, by the smoothness condition of the kernel $\mathcal{K}$, we can see that when $x,y\in \mathcal{B}$ and $(y_1,\dots,y_m)\in (\mathbb R^n)^m\setminus(2\mathcal{B})^m$,
\begin{equation}\label{regularity}
\begin{split}
&\Big|\mathcal{K}_t(x-y_1,\dots,x-y_m)-\mathcal{K}_t(y-y_1,\dots,y-y_m)\Big|\\
&=\frac{1}{t^{mn}}\bigg|\mathcal{K}\Big(\frac{x-y_1}{t},\dots,\frac{x-y_m}{t}\Big)-
\mathcal{K}\Big(\frac{y-y_1}{t},\dots,\frac{y-y_m}{t}\Big)\bigg|\\
&\lesssim\frac{t^{\delta}\cdot|x-y|^{\gamma}}{(t+\sum_{i=1}^m|x-y_i|)^{mn+\delta+\gamma}}.
\end{split}
\end{equation}
Consequently, by \eqref{regularity} and the vanishing condition of the kernel $\mathcal{K}$, we find that for any $x,y\in \mathcal{B}=B(x_0,r)$,
\begin{equation}\label{firstsecond2}
\begin{split}
&\Big|\mathcal{G}_t(\vec{f})(x)-\mathcal{G}_t(\vec{f})(y)\Big|\\
=&\bigg|\int_{(\mathbb R^n)^m}\Big[\mathcal{K}_t(x-y_1,\dots,x-y_m)-\mathcal{K}_t(y-y_1,\dots,y-y_m)\Big]
\bigg(\prod_{i=1}^m\big[f_i(y_i)-(f_i)_{2\mathcal{B}}\big]\bigg)\,dy_1\cdots dy_m\bigg|\\
\lesssim&\int_{(2\mathcal{B})^m}\Big[\big|\mathcal{K}_t(x-y_1,\dots,x-y_m)\big|+\big|\mathcal{K}_t(y-y_1,\dots,y-y_m)\big|\Big]
\prod_{i=1}^m\big|f_i(y_i)-(f_i)_{2\mathcal{B}}\big|\,dy_i\\
+&\int_{(\mathbb R^n)^m\setminus(2\mathcal{B})^m}
\frac{t^{\delta}\cdot|x-y|^{\gamma}}{(t+\sum_{i=1}^m|x-y_i|)^{mn+\delta+\gamma}}
\prod_{i=1}^m\big|f_i(y_i)-(f_i)_{2\mathcal{B}}\big|\,dy_i.
\end{split}
\end{equation}
In view of \eqref{A83}, the first term in \eqref{firstsecond2} is naturally controlled by
\begin{equation*}
\begin{split}
\int_{(2\mathcal{B})^m}\frac{1}{t^{mn}}\prod_{i=1}^m\big|f_i(y_i)-(f_i)_{2\mathcal{B}}\big|\,dy_i.
\end{split}
\end{equation*}
We now proceed to estimate the second term in \eqref{firstsecond2}. By a simple calculation, we can easily see that when $t>r$, $x\in \mathcal{B}$ and $(y_1,\dots,y_m)\in (\mathbb R^n)^m\setminus(2\mathcal{B})^m$,
\begin{equation}\label{keyestiw2}
t+\sum_{i=1}^m|x-y_i|\approx t+\sum_{i=1}^m|x_0-y_i|.
\end{equation}
Then the second term in \eqref{firstsecond2} is bounded by
\begin{equation*}
t^{\delta}\cdot\int_{(\mathbb R^n)^m\setminus(2\mathcal{B})^m}
\frac{(2r)^{\gamma}}{(t+\sum_{i=1}^m|x_0-y_i|)^{mn+\delta+\gamma}}
\prod_{i=1}^m\big|f_i(y_i)-(f_i)_{2\mathcal{B}}\big|\,dy_i.
\end{equation*}
Interchanging the order of integration, we further obtain
\begin{equation*}
\begin{split}
&\int_{r}^{\infty}\Big|\mathcal{G}_t(\vec{f})(x)-\mathcal{G}_t(\vec{f})(y)\Big|\frac{dt}{t}\\
&\lesssim\int_{(2\mathcal{B})^m}\prod_{i=1}^m\big|f_i(y_i)-(f_i)_{2\mathcal{B}}\big|\,dy_i
\bigg(\int_{r}^{\infty}\frac{1}{t^{mn+1}}\,dt\bigg)\\
&+\int_{(\mathbb R^n)^m\setminus(2\mathcal{B})^m}(2r)^{\gamma}\prod_{i=1}^m\big|f_i(y_i)-(f_i)_{2\mathcal{B}}\big|\,dy_i
\bigg(\int_{r}^{\infty}\frac{t^{\delta-1}}{(t+\sum_{i=1}^m|x_0-y_i|)^{mn+\delta+\gamma}}\,dt\bigg)\\
&\lesssim\prod_{i=1}^m\frac{1}{r^n}\int_{2\mathcal{B}}\big|f_i(y_i)-(f_i)_{2\mathcal{B}}\big|\,dy_i\\
&+\sum_{j=1}^{\infty}\int_{(2^{j+1}\mathcal{B})^{m}\setminus(2^j \mathcal{B})^{m}}
(2r)^{\gamma}\prod_{i=1}^m\big|f_i(y_i)-(f_i)_{2\mathcal{B}}\big|\,dy_i
\bigg(\int_{r}^{\infty}\frac{t^{\delta-1}}{(t+\sum_{i=1}^m|x_0-y_i|)^{mn+\delta+\gamma}}\,dt\bigg).
\end{split}
\end{equation*}
Note that when $t>r$ and $(y_1,\dots,y_m)\in(2^{j+1}\mathcal{B})^{m}\setminus(2^j \mathcal{B})^{m}$ with $j\in \mathbb{N}$, there exists a positive integer $1\leq k\leq m$ such that $y_k\in 2^{j+1}\mathcal{B}\setminus2^j \mathcal{B}$ and then
\begin{equation*}
|y_k-x_0|\geq 2^jr.
\end{equation*}
Consequently, we can deduce that
\begin{equation*}
\begin{split}
&\int_{r}^{\infty}\frac{t^{\delta-1}}{(t+\sum_{i=1}^m|x_0-y_i|)^{mn+\delta+\gamma}}\,dt
\leq\int_{r}^{\infty}\frac{t^{\delta-1}}{(t+|x_0-y_k|)^{mn+\delta+\gamma}}\,dt\\
&=\int_{r}^{|x_0-y_k|}\frac{t^{\delta-1}}{(t+|x_0-y_k|)^{mn+\delta+\gamma}}\,dt
+\int_{|x_0-y_k|}^{\infty}\frac{t^{\delta-1}}{(t+|x_0-y_k|)^{mn+\delta+\gamma}}\,dt\\
&\leq\int_{r}^{|x_0-y_k|}\frac{t^{\delta-1}}{(|x_0-y_k|)^{mn+\delta+\gamma}}\,dt+
\int_{|x_0-y_k|}^{\infty}\frac{t^{\delta-1}}{t^{mn+\delta+\gamma}}\,dt\\
&\leq\int_0^{|x_0-y_k|}\frac{t^{\delta-1}}{(|x_0-y_k|)^{mn+\delta+\gamma}}\,dt+
\int_{|x_0-y_k|}^{\infty}\frac{1}{t^{mn+\gamma+1}}\,dt\lesssim\frac{1}{(|x_0-y_k|)^{mn+\gamma}}.
\end{split}
\end{equation*}
Since $f_i\in \mathrm{BMO}(\mathbb R^n)$, $i=1,2,\dots,m$, from this and part (2) of Lemma \ref{BMOp}, it follows that for any $x,y\in \mathcal{B}=B(x_0,r)$,
\begin{equation*}
\begin{split}
&\int_{r}^{\infty}\Big|\mathcal{G}_t(\vec{f})(x)-\mathcal{G}_t(\vec{f})(y)\Big|\frac{dt}{t}\\
&\lesssim\prod_{i=1}^m\frac{1}{m(2\mathcal{B})}\int_{2\mathcal{B}}\big|f_i(y_i)-(f_i)_{2\mathcal{B}}\big|\,dy_i\\
&+\sum_{j=1}^{\infty}\int_{(2^{j+1}\mathcal{B})^{m}\setminus(2^j \mathcal{B})^{m}}
\frac{(2r)^{\gamma}}{(|x_0-y_k|)^{mn+\gamma}}\prod_{i=1}^m\big|f_i(y_i)-(f_i)_{2\mathcal{B}}\big|\,dy_i\\
&\lesssim\prod_{i=1}^m\big\|f_i\big\|_{\mathrm{BMO}}
+\sum_{j=1}^{\infty}\bigg\{\frac{(2r)^{\gamma}}{(2^j r)^{\gamma}}\cdot\prod_{i=1}^m\frac{1}{m(2^{j+1}\mathcal{B})}
\int_{2^{j+1} \mathcal{B}}\big|f_i(y_i)-(f_i)_{2\mathcal{B}}\big|\,dy_i\bigg\}\\
&\lesssim\prod_{i=1}^m\big\|f_i\big\|_{\mathrm{BMO}}
+\sum_{j=1}^{\infty}\Big(\frac{1}{2^{j-1}}\Big)^{\gamma}\times
\bigg\{\prod_{i=1}^mj\cdot\big\|f_i\big\|_{\mathrm{BMO}}\bigg\}\\
&=\prod_{i=1}^m\big\|f_i\big\|_{\mathrm{BMO}}
+\sum_{j=1}^{\infty}\frac{j^m}{2^{(j-1)\gamma}}\cdot\prod_{i=1}^m\big\|f_i\big\|_{\mathrm{BMO}}
\lesssim\prod_{i=1}^m\big\|f_i\big\|_{\mathrm{BMO}}.
\end{split}
\end{equation*}
Furthermore, it follows from the previous estimates that
\begin{equation*}
\begin{split}
\Big|\big[g_{\infty}(\vec{f})(x)\big]^2-\big[g_{\infty}(\vec{f})(y)\big]^2\Big|
\lesssim &\bigg[\prod_{i=1}^m\big\|f_i\big\|_{\mathrm{BMO}}\bigg]
\times\bigg[\prod_{i=1}^m\big\|f_i\big\|_{\mathrm{BMO}}\bigg]\\
\lesssim &\prod_{i=1}^m\big\|f_i\big\|^2_{\mathrm{BMO}},
\end{split}
\end{equation*}
where in the last step we have used Cauchy's inequality. Therefore,
\begin{equation*}
\begin{split}
I_{\infty}&=\frac{1}{|\mathcal{B}|}\int_{\mathcal{B}}
\underset{y\in\mathcal{B}}{\mathrm{ess\,sup}}\Big|\big[g_{\infty}(\vec{f})(x)\big]^2-\big[g_{\infty}(\vec{f})(y)\big]^2\Big|\,dx\\
&\lesssim\prod_{i=1}^m\big\|f_i\big\|^2_{\mathrm{BMO}}.
\end{split}
\end{equation*}
Combining the above estimates for both terms $I_0$ and $I_{\infty}$ yields the desired result \eqref{mainesti1}. This completes the proof of Theorem \ref{mainthm1}.
\end{proof}

\begin{proof}[Proof of Theorem $\ref{mainthm2}$]
Let $\vec{f}=(f_1,f_2,\dots,f_m)\in [\mathrm{BMO}(\mathbb R^n)]^{m}$. By the definition of $\mathrm{BLO}(\mathbb R^n)$, it suffices to prove that for any given ball $\mathcal{B}=B(x_0,r)\subset\mathbb R^n$ with center $x_0\in \mathbb R^n$ and radius $r\in(0,\infty)$, the following inequality holds:
\begin{equation}\label{mainesti2}
\frac{1}{m(\mathcal{B})}\int_{\mathcal{B}}
\Big[\big[S(\vec{f})(x)\big]^2-\underset{y\in\mathcal{B}}{\mathrm{ess\,inf}}\,\big[S(\vec{f})(y)\big]^2\Big]\,dx
\lesssim\prod_{i=1}^m\big\|f_i\big\|^2_{\mathrm{BMO}}.
\end{equation}
To prove \eqref{mainesti2}, we first decompose the integral defining $S(\vec{f})$ into two parts.
\begin{equation*}
\begin{split}
\big[S(\vec{f})(x)\big]^2&=\bigg(\int_0^{\infty}\int_{|z-x|<t}\big|\mathcal{G}_t(\vec{f})(z)\big|^2\frac{dzdt}{t^{n+1}}\bigg)\\
&=\int_0^{r}\int_{|z-x|<t}\big|\mathcal{G}_t(\vec{f})(z)\big|^2\frac{dzdt}{t^{n+1}}
+\int_{r}^{\infty}\int_{|z-x|<t}\big|\mathcal{G}_t(\vec{f})(z)\big|^2\frac{dzdt}{t^{n+1}}\\
&:=\big[S_0(\vec{f})(x)\big]^2+\big[S_{\infty}(\vec{f})(x)\big]^2.
\end{split}
\end{equation*}
Consequently, in view of \eqref{essinf}, we can deduce that
\begin{equation*}
\begin{split}
&\frac{1}{m(\mathcal{B})}\int_{\mathcal{B}}
\Big[\big[S(\vec{f})(x)\big]^2-\underset{y\in\mathcal{B}}{\mathrm{ess\,inf}}\,\big[S(\vec{f})(y)\big]^2\Big]\,dx\\
&=\frac{1}{m(\mathcal{B})}\int_{\mathcal{B}}\Big[\big[S_0(\vec{f})(x)\big]^2+\big[S_{\infty}(\vec{f})(x)\big]^2
-\underset{y\in\mathcal{B}}{\mathrm{ess\,inf}}\,\big[S(\vec{f})(y)\big]^2\Big]\,dx\\
&\leq\frac{1}{m(\mathcal{B})}\int_{\mathcal{B}}\Big[\big[S_0(\vec{f})(x)\big]^2+\big[S_{\infty}(\vec{f})(x)\big]^2
-\underset{y\in\mathcal{B}}{\mathrm{ess\,inf}}\,\big[S_{\infty}(\vec{f})(y)\big]^2\Big]\,dx\\
&\leq\frac{1}{m(\mathcal{B})}\int_{\mathcal{B}}\big[S_0(\vec{f})(x)\big]^2dx
+\frac{1}{m(\mathcal{B})}\int_{\mathcal{B}}
\underset{y\in\mathcal{B}}{\mathrm{ess\,sup}}\Big|\big[S_{\infty}(\vec{f})(x)\big]^2-\big[S_{\infty}(\vec{f})(y)\big]^2\Big|\,dx\\
&:=J_0+J_{\infty}.
\end{split}
\end{equation*}
Let us first consider the term $J_0$. For any $1\leq i\leq m$, we decompose the function $f_i$ as
\begin{equation*}
f_i=(f_i)_{4\mathcal{B}}+[f_i-(f_i)_{4\mathcal{B}}]\cdot\chi_{4\mathcal{B}}
+[f_i-(f_i)_{4\mathcal{B}}]\cdot\chi_{(4\mathcal{B})^{\complement}}:=f_i^1+f_i^2+f_i^3.
\end{equation*}
By equation \eqref{sum123} and the vanishing condition of the kernel $\mathcal{K}$, we thus obtain
\begin{equation*}
\begin{split}
&\mathcal{G}_t(\vec{f})(z)=\mathcal{G}_t(f_1,\dots,f_m)(z)\\
&=\sum_{\alpha_1,\dots,\alpha_m\in\{1,2,3\}}\int_{(\mathbb R^n)^m}\mathcal{K}_t(z-y_1,\dots,z-y_m)
f^{\alpha_1}_1(y_1)\cdots f^{\alpha_m}_m(y_m)\,dy_1\cdots dy_m\\
&=\sum_{\alpha_1,\dots,\alpha_m\in\{2,3\}}\mathcal{G}_t(f^{\alpha_1}_1,\dots,f^{\alpha_m}_m)(z),
\end{split}
\end{equation*}
and hence
\begin{equation*}
\begin{split}
J_0&=\frac{1}{m(\mathcal{B})}\int_{\mathcal{B}}\big[S_0(f_1,\dots,f_m)(x)\big]^2dx\\
&=\frac{1}{m(\mathcal{B})}\int_{\mathcal{B}}\big[S_0(f^2_1,\dots,f^2_m)(x)\big]^2dx+
\sum_{\alpha_1,\dots,\alpha_m\in \Xi}
\frac{1}{m(\mathcal{B})}\int_{\mathcal{B}}\big[S_0(f^{\alpha_1}_1,\dots,f^{\alpha_m}_m)(x)\big]^2dx\\
&:=J^{2,\dots,2}_0+\sum_{(\alpha_1,\dots,\alpha_m)\in\Xi}J^{\alpha_1,\dots,\alpha_m}_0,
\end{split}
\end{equation*}
where we denote
\begin{equation*}
\Xi:=\Big\{(\alpha_1,\dots,\alpha_m):\alpha_j\in\{2,3\},\mbox{there is at least one}~ \alpha_j\neq2,1\leq j\leq m\Big\}.
\end{equation*}
According to Theorem \ref{gs1}, we know that the $m$-linear operator $S$ is bounded from $L^{2m}(\mathbb R^n)\times\cdots\times L^{2m}(\mathbb R^n)$ into $L^2(\mathbb R^n)$. This fact, together with part (1) of Lemma \ref{BMOp}, implies that
\begin{equation*}
\begin{split}
J^{2,\dots,2}_0&\leq\frac{1}{m(\mathcal{B})}\big\|S(f^2_1,\dots,f^2_m)\big\|_{L^2}^2
\leq\frac{C}{m(\mathcal{B})}\bigg[\prod_{i=1}^m\big\|f^2_i\big\|_{L^{2m}}\bigg]^2\\
&=\frac{C}{m(\mathcal{B})}\bigg[\prod_{i=1}^m\bigg(\int_{4\mathcal{B}}
\big|f_i(y_i)-(f_i)_{4\mathcal{B}}\big|^{2m}dy_i\bigg)^{\frac{1}{2m}}\bigg]^2\\
&\lesssim\frac{1}{m(\mathcal{B})}\bigg[\prod_{i=1}^m\big\|f_i\big\|_{\mathrm{BMO}}m(4\mathcal{B})^{\frac{1}{2m}}\bigg]^2
\lesssim\prod_{i=1}^m\big\|f_i\big\|^2_{\mathrm{BMO}},
\end{split}
\end{equation*}
as desired. As in the proof of Theorem \ref{mainthm1}, we can also show that for an arbitrary point $z$ with $|z-x|<t$ and $0<t\leq r$,
\begin{equation*}
\begin{split}
&\big|\mathcal{G}_t\big(f^{\alpha_1}_1,\dots,f^{\alpha_m}_m\big)(z)\big|\\
&\leq t^{\delta}\int_{(\mathbb R^n)^{\ell}\setminus(4\mathcal{B})^{\ell}}\frac{1}{(t+\sum_{k=1}^{\ell}|z-y_k|)^{mn+\delta}}
\Big|\big[f_1(y_1)-(f_1)_{4\mathcal{B}}\big]\cdots\big[f_{\ell}(y_{\ell})-(f_{\ell})_{4\mathcal{B}}\big]\Big|\,dy_1\cdots dy_{\ell}\\
&\times\prod_{k=\ell+1}^m\int_{4\mathcal{B}}\big|f_k(y_k)-(f_k)_{4\mathcal{B}}\big|\,dy_k,
\end{split}
\end{equation*}
when $1\leq\ell<m$ and
\begin{equation*}
\alpha_1=\cdots=\alpha_{\ell}=3\qquad \&\qquad \alpha_{\ell+1}=\cdots=\alpha_m=2.
\end{equation*}
It is easy to check that when $x\in \mathcal{B}=B(x_0,r)$, $y\in 2^{j+1}\mathcal{B}\setminus 2^j\mathcal{B}$ with $j\geq2$, $|z-x|<t$ and $0<t\leq r$,
\begin{equation}\label{factw2}
|z-y|\approx |x_0-y|.
\end{equation}
From \eqref{factw2} and \eqref{equaw1}, it then follows that
\begin{equation*}
\begin{split}
&\big|\mathcal{G}_t\big(f^{\alpha_1}_1,\dots,f^{\alpha_m}_m\big)(z)\big|\\
&\leq t^{\delta}\sum_{j=2}^{\infty}\int_{(2^{j+1}\mathcal{B})^{\ell}\setminus(2^j \mathcal{B})^{\ell}}\frac{1}{(\sum_{k=1}^{\ell}|z-y_k|)^{mn+\delta}}
\Big|\big[f_1(y_1)-(f_1)_{4\mathcal{B}}\big]\cdots\big[f_{\ell}(y_{\ell})-(f_{\ell})_{4\mathcal{B}}\big]\Big|\,dy_1\cdots dy_{\ell}\\
&\times\prod_{k=\ell+1}^m\int_{4\mathcal{B}}\big|f_k(y_k)-(f_k)_{4\mathcal{B}}\big|\,dy_k\\
&\lesssim t^{\delta}\sum_{j=2}^{\infty}\bigg\{\prod_{k=1}^{\ell}
\int_{2^{j+1}\mathcal{B}\setminus 2^j\mathcal{B}}
\frac{1}{(|x_0-y_k|)^{\frac{mn+\delta}{\ell}}}\big|f_k(y_k)-(f_k)_{4\mathcal{B}}\big|\,dy_k\bigg\}\\
&\times\prod_{k=\ell+1}^m\Big[\big\|f_k\big\|_{\mathrm{BMO}}\cdot m\big(4\mathcal{B}\big)\Big]\\
&\lesssim t^{\delta}\sum_{j=2}^{\infty}\bigg\{\prod_{k=1}^{\ell}\frac{1}{m(2^j\mathcal{B})^{\frac{mn+\delta}{\ell n}}}
\int_{2^{j+1}\mathcal{B}}\big|f_k(y_k)-(f_k)_{4\mathcal{B}}\big|\,dy_k\bigg\}
\times\prod_{k=\ell+1}^m\Big[\big\|f_k\big\|_{\mathrm{BMO}}\cdot m\big(4\mathcal{B}\big)\Big].
\end{split}
\end{equation*}
We can now argue exactly as we did in the proof of Theorem \ref{mainthm1} to get
\begin{equation*}
\big|\mathcal{G}_t\big(f^{\alpha_1}_1,\dots,f^{\alpha_m}_m\big)(z)\big|
\lesssim t^{\delta}\sum_{j=2}^{\infty}\bigg\{\frac{1}{m(2^j\mathcal{B})^{\frac{\delta}{n}}}\cdot j^{\ell}\bigg\}
\times\prod_{k=1}^m\big\|f_k\big\|_{\mathrm{BMO}}.
\end{equation*}
Therefore,
\begin{equation*}
\begin{split}
J^{\alpha_1,\dots,\alpha_m}_0
&=\frac{1}{m(\mathcal{B})}\int_{\mathcal{B}}\big[S_0(f^{\alpha_1}_1,\dots,f^{\alpha_m}_m)(x)\big]^2dx\\
&\lesssim\frac{1}{m(\mathcal{B})}\int_{\mathcal{B}}\bigg(\int_0^{r}\int_{|z-x|<t}t^{2\delta}\frac{dzdt}{t^{n+1}}\bigg)
\bigg\{\sum_{j=2}^{\infty}\frac{j^{\ell}}{m(2^j\mathcal{B})^{\frac{\delta}{n}}}\bigg\}^2
\times\prod_{k=1}^m\big\|f_k\big\|^2_{\mathrm{BMO}}\\
&\lesssim\frac{1}{m(\mathcal{B})}\int_{\mathcal{B}}\bigg(\int_0^{r}t^{2\delta-1}dt\bigg)
\bigg\{\sum_{j=1}^{\infty}\frac{j^{\ell}}{(2^j)^{\delta}r^{\delta}}\bigg\}^2
\times\prod_{k=1}^m\big\|f_k\big\|^2_{\mathrm{BMO}}\\
&\lesssim\prod_{k=1}^m\big\|f_k\big\|^2_{\mathrm{BMO}}.
\end{split}
\end{equation*}
Let us now consider the remaining case when $\alpha_1=\cdots=\alpha_m=3$. By using the same arguments as in Theorem \ref{mainthm1}, we can deduce that
\begin{equation*}
\begin{split}
&\big|\mathcal{G}_t\big(f^{3}_1,\dots,f^{3}_{m}\big)(z)\big|\\
&\lesssim\int_{(\mathbb R^n)^{m}\setminus (4\mathcal{B})^{m}}\frac{t^{\delta}}{(t+\sum_{k=1}^{m}|z-y_k|)^{mn+\delta}}\\
&\times\Big|\big[f_1(y_1)-(f_1)_{4\mathcal{B}}\big]\cdots\big[f_{m}(y_{m})-(f_{m})_{4\mathcal{B}}\big]\Big|\,dy_1\cdots dy_{m}\\
&\leq\sum_{j=2}^{\infty}\int_{(2^{j+1}\mathcal{B})^{m}\setminus(2^j \mathcal{B})^{m}}
\frac{t^{\delta}}{(t+\sum_{k=1}^m|z-y_k|)^{mn+\delta}}\\
&\times\Big|\big[f_1(y_1)-(f_1)_{4\mathcal{B}}\big]\cdots\big[f_{m}(y_{m})-(f_{m})_{4\mathcal{B}}\big]\Big|\,dy_1\cdots dy_{m}\\
&\lesssim t^{\delta}\sum_{j=2}^{\infty}\bigg\{\prod_{k=1}^{m}
\int_{2^{j+1}\mathcal{B}\setminus 2^j\mathcal{B}}
\frac{1}{(|x_0-y_k|)^{\frac{mn+\delta}{m}}}\big|f_k(y_k)-(f_k)_{4\mathcal{B}}\big|\,dy_k\bigg\}\\
&\lesssim t^{\delta}\sum_{j=2}^{\infty}\bigg\{\prod_{k=1}^{m}\frac{1}{m(2^j\mathcal{B})^{\frac{mn+\delta}{m n}}}
\int_{2^{j+1}\mathcal{B}}\big|f_k(y_k)-(f_k)_{4\mathcal{B}}\big|\,dy_k\bigg\},
\end{split}
\end{equation*}
where in the last inequality we have used \eqref{factw2} and \eqref{equaw2}. Hence, as in the proof of Theorem \ref{mainthm1}, we can also prove the following result.
\begin{equation*}
\big|\mathcal{G}_t\big(f^{3}_1,\dots,f^{3}_m\big)(z)\big|
\lesssim t^{\delta}\sum_{j=2}^{\infty}\bigg\{\frac{1}{m(2^j\mathcal{B})^{\frac{\delta}{n}}}\cdot j^{m}\bigg\}
\times\prod_{k=1}^m\big\|f_k\big\|_{\mathrm{BMO}}.
\end{equation*}
Therefore, we have
\begin{equation*}
\begin{split}
J^{3,\dots,3}_0
&=\frac{1}{m(\mathcal{B})}\int_{\mathcal{B}}\big[S_0(f^{3}_1,\dots,f^{3}_m)(x)\big]^2dx\\
&\lesssim\frac{1}{m(\mathcal{B})}\int_{\mathcal{B}}\bigg(\int_0^{r}\int_{|z-x|<t}t^{2\delta}\frac{dzdt}{t^{n+1}}\bigg)
\bigg\{\sum_{j=1}^{\infty}\frac{j^{m}}{m(2^j\mathcal{B})^{\frac{\delta}{n}}}\bigg\}^2
\times\prod_{k=1}^m\big\|f_k\big\|^2_{\mathrm{BMO}}\\
&\lesssim\frac{1}{m(\mathcal{B})}\int_{\mathcal{B}}\bigg(\int_0^{r}t^{2\delta-1}dt\bigg)
\bigg\{\sum_{j=1}^{\infty}\frac{j^{m}}{(2^j)^{\delta}r^{\delta}}\bigg\}^2
\times\prod_{k=1}^m\big\|f_k\big\|^2_{\mathrm{BMO}}\\
&\lesssim\prod_{k=1}^m\big\|f_k\big\|^2_{\mathrm{BMO}}.
\end{split}
\end{equation*}
Summing up the above estimates, we conclude that
\begin{equation*}
J_0\lesssim\prod_{i=1}^m\big\|f_i\big\|^2_{\mathrm{BMO}}.
\end{equation*}
Let us now turn to deal with the other term $J_{\infty}$. We first claim that for any $x\in\mathcal{B}=B(x_0,r)$ and $z\in\mathbb R^n$ satisfying $|z|<t$ and $t>r$,
\begin{equation}\label{keyestiw5}
\big|\mathcal{G}_t(\vec{f})(x+z)\big|\lesssim\prod_{i=1}^m\big\|f_i\big\|_{\mathrm{BMO}}.
\end{equation}
In fact, one can easily check that the same proof of $I_{\infty}$ above goes along in this more general situation (some easy modifications). We shall repeat the argument here for completeness. Notice that $t>r$, then there exists a nonnegative integer $k\in \mathbb{N}\cup\{0\}$ such that $2^{k}r<t\leq 2^{k+1}r$. By the vanishing condition and size condition of the kernel $\mathcal{K}$ and \eqref{A83}, we have
\begin{equation}\label{firstsecondS1}
\begin{split}
&\big|\mathcal{G}_t(\vec{f})(x+z)\big|\\
&=\bigg|\int_{(\mathbb R^n)^m}\mathcal{K}_t(x+z-y_1,\dots,x+z-y_m)
\bigg(\prod_{i=1}^m\big[f_i(y_i)-(f_i)_{2^{k+2}\mathcal{B}}\big]\bigg)\,dy_1\cdots dy_m\bigg|\\
&\lesssim\int_{(2^{k+2}\mathcal{B})^m}\frac{1}{t^{mn}}\prod_{i=1}^m\big|f_i(y_i)-(f_i)_{2^{k+2}\mathcal{B}}\big|\,dy_i\\
&+\int_{(\mathbb R^n)^m\setminus(2^{k+2}\mathcal{B})^m}
\frac{t^{\delta}}{(t+\sum_{i=1}^m|x+z-y_i|)^{mn+\delta}}
\prod_{i=1}^m\big|f_i(y_i)-(f_i)_{2^{k+2}\mathcal{B}}\big|\,dy_i.
\end{split}
\end{equation}
Clearly, the first term in \eqref{firstsecondS1} is dominated by
\begin{equation*}
\begin{split}
&\prod_{i=1}^m\frac{1}{(2^{k}r)^n}\int_{2^{k+2}\mathcal{B}}\big|f_i(y_i)-(f_i)_{2^{k+2}\mathcal{B}}\big|\,dy_i\\
&\lesssim\prod_{i=1}^m\frac{1}{m(2^{k+2}\mathcal{B})}\int_{2^{k+2}\mathcal{B}}\big|f_i(y_i)-(f_i)_{2^{k+2}\mathcal{B}}\big|\,dy_i
\leq\prod_{i=1}^m\big\|f_i\big\|_{\mathrm{BMO}}.
\end{split}
\end{equation*}
Observe that when $x\in \mathcal{B}$, $|z|<t$ and $(y_1,\dots,y_m)\in (2^{j+1}\mathcal{B})^{m}\setminus(2^j\mathcal{B})^{m}$ with $j\geq k+2$ and $k\in \mathbb{N}\cup\{0\}$, one has
\begin{equation*}
t+\sum_{i=1}^m|x+z-y_i|\approx t+\sum_{i=1}^m|x-y_i|,
\end{equation*}
and
\begin{equation*}
\sum_{i=1}^m|x-y_i|\geq\max_{1\leq k\leq m}|x-y_k|\geq 2^{j}r\cong m(2^j \mathcal{B})^{1/n}.
\end{equation*}
Hence, the second term in \eqref{firstsecondS1} is bounded by
\begin{equation*}
\begin{split}
&\sum_{j=k+2}^{\infty}\big(2^{k+1}r\big)^{\delta}\int_{(2^{j+1}\mathcal{B})^{m}\setminus(2^j \mathcal{B})^{m}}
\frac{1}{(\sum_{i=1}^m|x-y_i|)^{mn+\delta}}\prod_{i=1}^m\big|f_i(y_i)-(f_i)_{2^{k+2}\mathcal{B}}\big|\,dy_i\\
&\lesssim\sum_{j=k+2}^{\infty}\big(2^{k+1}r\big)^{\delta}
\bigg\{\prod_{i=1}^m\frac{1}{m(2^{j+1}\mathcal{B})^{\frac{mn+\delta}{mn}}}
\int_{2^{j+1} \mathcal{B}}\big|f_i(y_i)-(f_i)_{2^{k+2}\mathcal{B}}\big|\,dy_i\bigg\}.
\end{split}
\end{equation*}
Moreover, in view of part (2) of Lemma \ref{BMOp}, the above expression is further bounded by
\begin{equation*}
\begin{split}
&\sum_{j=k+2}^{\infty}\frac{(2^{k+1}r)^{\delta}}{(2^{j+1}r)^{\delta}}
\bigg\{\prod_{i=1}^m\frac{1}{m(2^{j+1}\mathcal{B})}
\int_{2^{j+1} \mathcal{B}}\big|f_i(y_i)-(f_i)_{2^{k+2}\mathcal{B}}\big|\,dy_i\bigg\}\\
&\lesssim\sum_{j=k+2}^{\infty}\frac{1}{(2^{j-k})^{\delta}}\bigg\{\prod_{i=1}^m(j-k)\cdot\big\|f_i\big\|_{\mathrm{BMO}}\bigg\}\\
&=\sum_{j=2}^{\infty}\frac{j^m}{2^{j\delta}}\cdot\prod_{i=1}^m\big\|f_i\big\|_{\mathrm{BMO}}
\lesssim\prod_{i=1}^m\big\|f_i\big\|_{\mathrm{BMO}}.
\end{split}
\end{equation*}
Combining the above estimates for both terms in \eqref{firstsecondS1} implies our desired estimate \eqref{keyestiw5}. Consequently, by using the triangle inequality and \eqref{keyestiw5}, we can see that for any $x,y\in \mathcal{B}=B(x_0,r)$,
\begin{equation*}
\begin{split}
&\Big|\big[S_{\infty}(\vec{f})(x)\big]^2-\big[S_{\infty}(\vec{f})(y)\big]^2\Big|\\
&=\bigg|\int_{r}^{\infty}\int_{|z|<t}\big|\mathcal{G}_t(\vec{f})(x+z)\big|^2-\big|\mathcal{G}_t(\vec{f})(y+z)\big|^2\frac{dzdt}{t^{n+1}}\bigg|\\
&\leq\int_{r}^{\infty}\int_{|z|<t}\Big[\big|\mathcal{G}_t(\vec{f})(x+z)\big|+\big|\mathcal{G}_t(\vec{f})(y+z)\big|\Big]
\cdot\Big|\mathcal{G}_t(\vec{f})(x+z)-\mathcal{G}_t(\vec{f})(y+z)\Big|\frac{dzdt}{t^{n+1}}\\
&\lesssim\prod_{i=1}^m\big\|f_i\big\|_{\mathrm{BMO}}
\times\int_{r}^{\infty}\int_{|z|<t}\Big|\mathcal{G}_t(\vec{f})(x+z)-\mathcal{G}_t(\vec{f})(y+z)\Big|\frac{dzdt}{t^{n+1}}.
\end{split}
\end{equation*}
On the other hand, by the smoothness condition of the kernel $\mathcal{K}$, we can see that for any $x,y\in \mathcal{B}$, $|z|<t$ and $(y_1,\dots,y_m)\in (\mathbb R^n)^m\setminus(4\mathcal{B})^m$,
\begin{align}\label{regularity3}
&\Big|\mathcal{K}_t(x+z-y_1,\dots,x+z-y_m)-\mathcal{K}_t(y+z-y_1,\dots,y+z-y_m)\Big|\notag\\
&=\frac{1}{t^{mn}}\bigg|\mathcal{K}\Big(\frac{x+z-y_1}{t},\dots,\frac{x+z-y_m}{t}\Big)-
\mathcal{K}\Big(\frac{y+z-y_1}{t},\dots,\frac{y+z-y_m}{t}\Big)\bigg|\notag\\
&\lesssim\frac{t^{\delta}\cdot|x-y|^{\gamma}}{(t+\sum_{i=1}^m|x+z-y_i|)^{mn+\delta+\gamma}}.
\end{align}
Thus, by \eqref{regularity3} and the vanishing condition of the kernel $\mathcal{K}$, we obtain that for any $x,y\in \mathcal{B}=B(x_0,r)$,
\begin{equation}\label{firstsecond2S2}
\begin{split}
&\Big|\mathcal{G}_t(\vec{f})(x+z)-\mathcal{G}_t(\vec{f})(y+z)\Big|\\
=&\bigg|\int_{(\mathbb R^n)^m}\Big[\mathcal{K}_t(x+z-y_1,\dots,x+z-y_m)-\mathcal{K}_t(y+z-y_1,\dots,y+z-y_m)\Big]\\
&\times\bigg(\prod_{i=1}^m\big[f_i(y_i)-(f_i)_{4\mathcal{B}}\big]\bigg)\,dy_1\cdots dy_m\bigg|\\
\lesssim&\int_{(4\mathcal{B})^m}\Big[\big|\mathcal{K}_t(x+z-y_1,\dots,x+z-y_m)\big|+\big|\mathcal{K}_t(y+z-y_1,\dots,y+z-y_m)\big|\Big]\\
&\times\prod_{i=1}^m\big|f_i(y_i)-(f_i)_{4\mathcal{B}}\big|\,dy_i\\
+&\int_{(\mathbb R^n)^m\setminus(4\mathcal{B})^m}
\frac{t^{\delta}\cdot|x-y|^{\gamma}}{(t+\sum_{i=1}^m|x+z-y_i|)^{mn+\delta+\gamma}}
\prod_{i=1}^m\big|f_i(y_i)-(f_i)_{4\mathcal{B}}\big|\,dy_i.
\end{split}
\end{equation}
By using \eqref{A83}, the first term in \eqref{firstsecond2S2} is naturally controlled by
\begin{equation*}
\begin{split}
\int_{(4\mathcal{B})^m}\frac{1}{t^{mn}}\prod_{i=1}^m\big|f_i(y_i)-(f_i)_{4\mathcal{B}}\big|\,dy_i.
\end{split}
\end{equation*}
Observe that for any $|z|<t$ with $t>r$, $x\in\mathcal{B}$ and $(y_1,\dots,y_m)\in (\mathbb R^n)^m\setminus(4\mathcal{B})^m$,
\begin{equation*}
t+\sum_{i=1}^m|x+z-y_i|\approx t+\sum_{i=1}^m|x-y_i|.
\end{equation*}
Then the second term in \eqref{firstsecond2S2} is bounded by
\begin{equation*}
t^{\delta}\cdot\int_{(\mathbb R^n)^m\setminus(4\mathcal{B})^m}
\frac{(2r)^{\gamma}}{(t+\sum_{i=1}^m|x-y_i|)^{mn+\delta+\gamma}}
\prod_{i=1}^m\big|f_i(y_i)-(f_i)_{4\mathcal{B}}\big|\,dy_i.
\end{equation*}
Interchanging the order of integration in the following calculation, we have
\begin{equation*}
\begin{split}
&\int_{r}^{\infty}\int_{|z|<t}\Big|\mathcal{G}_t(\vec{f})(x+z)-\mathcal{G}_t(\vec{f})(y+z)\Big|\frac{dzdt}{t^{n+1}}\\
&\lesssim\int_{(4\mathcal{B})^m}\prod_{i=1}^m\big|f_i(y_i)-(f_i)_{4\mathcal{B}}\big|\,dy_i
\bigg(\int_{r}^{\infty}\int_{|z|<t}\frac{1}{t^{mn+n+1}}dzdt\bigg)\\
&+\int_{(\mathbb R^n)^m\setminus(4\mathcal{B})^m}(2r)^{\gamma}\prod_{i=1}^m\big|f_i(y_i)-(f_i)_{4\mathcal{B}}\big|\,dy_i
\bigg(\int_{r}^{\infty}\int_{|z|<t}\frac{t^{\delta}}{(t+\sum_{i=1}^m|x-y_i|)^{mn+\delta+\gamma}}\frac{dzdt}{t^{n+1}}\bigg)\\
&\lesssim\int_{(4\mathcal{B})^m}\prod_{i=1}^m\big|f_i(y_i)-(f_i)_{4\mathcal{B}}\big|\,dy_i
\bigg(\int_{r}^{\infty}\frac{1}{t^{mn+1}}dt\bigg)\\
&+\int_{(\mathbb R^n)^m\setminus(4\mathcal{B})^m}(2r)^{\gamma}\prod_{i=1}^m\big|f_i(y_i)-(f_i)_{4\mathcal{B}}\big|\,dy_i
\bigg(\int_{r}^{\infty}\frac{t^{\delta-1}}{(t+\sum_{i=1}^m|x-y_i|)^{mn+\delta+\gamma}}dt\bigg)\\
&\lesssim\prod_{i=1}^m\frac{1}{m(4\mathcal{B})}\int_{4\mathcal{B}}\big|f_i(y_i)-(f_i)_{4\mathcal{B}}\big|\,dy_i\\
&+\sum_{j=2}^{\infty}\int_{(2^{j+1}\mathcal{B})^{m}\setminus(2^j \mathcal{B})^{m}}
(2r)^{\gamma}\prod_{i=1}^m\big|f_i(y_i)-(f_i)_{4\mathcal{B}}\big|\,dy_i
\bigg(\int_{r}^{\infty}\frac{t^{\delta-1}}{(t+\sum_{i=1}^m|x-y_i|)^{mn+\delta+\gamma}}dt\bigg).
\end{split}
\end{equation*}
We now proceed exactly as in Theorem \ref{mainthm1}, and obtain
\begin{equation*}
\begin{split}
&\int_{r}^{\infty}\int_{|z|<t}\Big|\mathcal{G}_t(\vec{f})(x+z)-\mathcal{G}_t(\vec{f})(y+z)\Big|\frac{dzdt}{t^{n+1}}\\
&\lesssim\prod_{i=1}^m\big\|f_i\big\|_{\mathrm{BMO}}
+\sum_{j=2}^{\infty}\frac{j^m}{2^{(j-1)\gamma}}\cdot\prod_{i=1}^m\big\|f_i\big\|_{\mathrm{BMO}}
\lesssim\prod_{i=1}^m\big\|f_i\big\|_{\mathrm{BMO}}.
\end{split}
\end{equation*}
Furthermore, it follows from the previous estimates that
\begin{equation*}
\begin{split}
\Big|\big[S_{\infty}(\vec{f})(x)\big]^2-\big[S_{\infty}(\vec{f})(y)\big]^2\Big|
\lesssim &\bigg[\prod_{i=1}^m\big\|f_i\big\|_{\mathrm{BMO}}\bigg]
\times\bigg[\prod_{i=1}^m\big\|f_i\big\|_{\mathrm{BMO}}\bigg]\\
\lesssim &\prod_{i=1}^m\big\|f_i\big\|^2_{\mathrm{BMO}},
\end{split}
\end{equation*}
where the last inequality follows from Cauchy's inequality. Hence,
\begin{equation*}
\begin{split}
J_{\infty}&=\frac{1}{|\mathcal{B}|}\int_{\mathcal{B}}
\underset{y\in\mathcal{B}}{\mathrm{ess\,sup}}\Big|\big[S_{\infty}(\vec{f})(x)\big]^2-\big[S_{\infty}(\vec{f})(y)\big]^2\Big|\,dx\\
&\lesssim\prod_{i=1}^m\big\|f_i\big\|^2_{\mathrm{BMO}}.
\end{split}
\end{equation*}
Combining the above estimates for both terms $J_0$ and $J_{\infty}$ yields the desired result \eqref{mainesti2}. This concludes the proof of Theorem \ref{mainthm2}.
\end{proof}

Let $\delta$ and $\gamma$ be the same as in Definition \ref{defin12}. By using the same procedure as in the proofs of Theorems \ref{mainthm1} and \ref{mainthm2}, and invoking Theorem \ref{gs1ambda}, we are able to show that if $g^{\ast}_{\lambda}(\vec{f})(x_0)<+\infty$ for a single point $x_0\in\mathbb R^n$ and $\vec{f}\in[\mathrm{BMO}(\mathbb R^n)]^{m}$, then $g^{\ast}_{\lambda}(\vec{f})(x)$ is finite almost everywhere in $\mathbb R^n$. Moreover, the multilinear Littlewood--Paley $g^{\ast}_{\lambda}$-function is bounded from $\mathrm{BMO}(\mathbb R^n)\times\cdots\times \mathrm{BMO}(\mathbb R^n)$ into $\mathrm{BLO}(\mathbb R^n)$.

\begin{thm}\label{mainthm3}
Suppose that $\lambda>3m+{(2\delta+2\gamma)}/n$ with $2\leq m\in \mathbb{N}$ and $\gamma,\delta>0$. For any $\vec{f}=(f_1,f_2,\dots,f_m)\in [\mathrm{BMO}(\mathbb R^n)]^{m}$, then $g^{\ast}_{\lambda}(\vec{f})(x)$ is either infinite everywhere or finite almost everywhere, and in the latter case, we have
\begin{equation*}
\big\|\big[g^{\ast}_{\lambda}(\vec{f})\big]^2\big\|_{\mathrm{BLO}}
\lesssim\prod_{i=1}^m\big\|f_i\big\|^2_{\mathrm{BMO}}.
\end{equation*}
\end{thm}

As a direct consequence of \eqref{BLOsquare} and Theorem \ref{mainthm3}, we have
\begin{cor}
Suppose that $\lambda>3m+{(2\delta+2\gamma)}/n$ with $2\leq m\in \mathbb{N}$ and $\delta,\gamma>0$. For any $\vec{f}=(f_1,f_2,\dots,f_m)\in [\mathrm{BMO}(\mathbb R^n)]^{m}$, then $g^{\ast}_{\lambda}(\vec{f})(x)$ is either infinite everywhere or finite almost everywhere, and in the latter case, we have
\begin{equation*}
\big\|g^{\ast}_{\lambda}(\vec{f})\big\|_{\mathrm{BLO}}\lesssim\prod_{i=1}^m\big\|f_i\big\|_{\mathrm{BMO}}.
\end{equation*}
\end{cor}
We now consider the multilinear analogues of Leckband's result for $g$-function, Lusin's area integral, and Littlewood--Paley $g^{\ast}_{\lambda}$-function. Let $\mathcal{T}_{g}(\vec{f})$ be one of the multilinear operators $g(\vec{f})$, $S(\vec{f})$ and $g^{\ast}_{\lambda}(\vec{f})$ for $\lambda>2m$, by using similar arguments to those in \cite{he}, \cite{leck} and \cite{sun}, we can show that $\mathcal{T}_{g}(\vec{f})(x)$ is finite everywhere in $\mathbb R^n$, and bounded from $L^{\infty}(\mathbb R^n)\times\cdots\times L^{\infty}(\mathbb R^n)$ into $\mathrm{BLO}(\mathbb R^n)$, in view of the relation \eqref{relation11}.

When $f_i\in L^{\infty}(\mathbb R^n)$ for $i=1,2,\dots,m$, we denote simply by
\begin{equation*}
\vec{f}:=(f_1,f_2,\dots,f_m)\in [L^{\infty}(\mathbb R^n)]^{m}.
\end{equation*}

We can deduce the following results.
\begin{thm}\label{mainthm11}
For any $\vec{f}=(f_1,f_2,\dots,f_m)\in [L^{\infty}(\mathbb R^n)]^{m}$ and $2\leq m\in \mathbb{N}$, then $g(\vec{f})(x)$ is finite everywhere, and there exists a positive constant $C$, independent of $\vec{f}$, such that
\begin{equation*}
\big\|\big[g(\vec{f})\big]^2\big\|_{\mathrm{BLO}}\leq C\prod_{i=1}^m\big\|f_i\big\|^2_{L^{\infty}},
\end{equation*}
and hence
\begin{equation*}
\big\|g(\vec{f})\big\|_{\mathrm{BLO}}\leq C\prod_{i=1}^m\big\|f_i\big\|_{L^{\infty}}.
\end{equation*}
\end{thm}

\begin{thm}\label{mainthm12}
For any $\vec{f}=(f_1,f_2,\dots,f_m)\in [L^{\infty}(\mathbb R^n)]^{m}$ and $2\leq m\in \mathbb{N}$, then $S(\vec{f})(x)$ is finite everywhere, and there exists a positive constant $C$, independent of $\vec{f}$, such that
\begin{equation*}
\big\|\big[S(\vec{f})\big]^2\big\|_{\mathrm{BLO}}\leq C\prod_{i=1}^m\big\|f_i\big\|^2_{L^{\infty}},
\end{equation*}
and hence
\begin{equation*}
\big\|S(\vec{f})\big\|_{\mathrm{BLO}}\leq C\prod_{i=1}^m\big\|f_i\big\|_{L^{\infty}}.
\end{equation*}
\end{thm}

\begin{thm}\label{mainthm13}
Suppose that $\lambda>3m+{(2\delta+2\gamma)}/n$ and $\delta,\gamma>0$. For any $\vec{f}=(f_1,f_2,\dots,f_m)\in [L^{\infty}(\mathbb R^n)]^{m}$ and $2\leq m\in \mathbb{N}$, then $g^{\ast}_{\lambda}(\vec{f})(x)$ is finite everywhere, and there exists a positive constant $C$, independent of $\vec{f}$, such that
\begin{equation*}
\big\|\big[g^{\ast}_{\lambda}(\vec{f})\big]^2\big\|_{\mathrm{BLO}}\leq C\prod_{i=1}^m\big\|f_i\big\|^2_{L^{\infty}},
\end{equation*}
and hence
\begin{equation*}
\big\|g^{\ast}_{\lambda}(\vec{f})\big\|_{\mathrm{BLO}}\leq C\prod_{i=1}^m\big\|f_i\big\|_{L^{\infty}}.
\end{equation*}
\end{thm}

\section{Concluding remarks}
In the last section, we point out that our arguments may be extended to the case where the kernels of multilinear Littlewood--Paley operators are of non-convolution type, and the conclusions of our main theorems remain true in this context. In 2015, Xue and Yan defined and studied the multilinear Littlewood--Paley operators with non-convolution type kernels, including multilinear $g$-function, Lusin's area integral and Littlewood--Paley $g^{\ast}_{\lambda}$-function. By using similar arguments, we can also obtain the existence and boundedness of multilinear Littlewood--Paley operators with non-convolution type kernels on products of BMO spaces (BMO--BLO results).

Let us give the definition of the multilinear Littlewood--Paley kernel (of non-convolution type).
\begin{defin}[\cite{xueqing}]
Let $K(x,y_1,\dots,y_m)$ be a locally integrable function defined away from the diagonal $x=y_1=\cdots=y_m$ in $(\mathbb R^n)^{m+1}$. We say that a function $K(x,y_1,\dots,y_m)$ defined on $(\mathbb R^n)^{m+1}$ is a multilinear Littlewood--Paley kernel (of non-convolution type), if the following three conditions are satisfied.
\begin{enumerate}
  \item (\textbf{The vanishing condition}): for all $x\in \mathbb R^n$,
  \begin{equation*}
  \int_{\mathbb R^n}K(x,y_1,\dots,y_i,\dots,y_m)\,dy_i=0,\quad \mbox{for}~~ i=1,2,\dots,m;
  \end{equation*}
  \item (\textbf{the size condition}): for some positive constants $C$ and $\delta$,
  \begin{equation*}
  \big|K(x,y_1,y_2,\dots,y_m)\big|\leq C\cdot\frac{1}{(1+\sum_{j=1}^m|y_j|)^{mn+\delta}};
  \end{equation*}
  \item (\textbf{the smoothness condition}): for some positive constants $C$ and $\gamma$,
  \begin{equation*}
  \big|K(x,y_1,\dots,y_i,\dots,y_m)-K(x,y_1,\dots,y'_i,\dots,y_m)\big|
  \leq C\cdot\frac{|y_i-y_i'|^{\gamma}}{(1+\sum_{j=1}^m|x-y_j|)^{mn+\delta+\gamma}}
  \end{equation*}
  whenever $2|y_i-y_i'|\leq |x-y_i|$ for all $1\leq i\leq m$, and
  \begin{equation*}
  \big|K(x,y_1,y_2,\dots,y_m)-K(x',y_1,y_2,\dots,y_m)\big|
  \leq C\cdot\frac{|x-x'|^{\gamma}}{(1+\sum_{j=1}^m|x-y_j|)^{mn+\delta+\gamma}}
  \end{equation*}
  whenever $2|x-x'|\leq\max_{1\leq j\leq m}|x-y_j|$.
\end{enumerate}
\end{defin}

\begin{defin}[\cite{xueqing}]
For any $\vec{f}=(f_1,\dots,f_m)\in \overbrace{\mathcal{S}(\mathbb R^n)\times\cdots\times \mathcal{S}(\mathbb R^n)}^m$ and any $t>0$, we denote
\begin{equation*}
\mathcal{K}_t(x,y_1,y_2,\dots,y_m):=\frac{1}{t^{mn}}\mathcal{K}\Big(\frac{x}{\,t\,},\frac{y_1}{t},\frac{y_2}{t},\dots,\frac{y_m}{t}\Big),
\end{equation*}
and
\begin{equation*}
\mathcal{G}_t(\vec{f})(x):=\int_{(\mathbb R^n)^m}\mathcal{K}_t(x,y_1,y_2,\dots,y_m)\prod_{i=1}^m f_i(y_i)\,dy_i,
\quad \mbox{for all}\;\,x\notin\bigcap_{i=1}^m \mathrm{supp}\, f_i.
\end{equation*}
Then the multilinear Littlewood--Paley $g$-function, multilinear Lusin's area integral and multilinear Littlewood--Paley $g^{\ast}_{\lambda}$-function with non-convolution type kernels are defined, respectively, by
\begin{equation*}
g'(\vec{f})(x):=\bigg(\int_0^{\infty}\big|\mathcal{G}_t(\vec{f})(x)\big|^2\frac{dt}{t}\bigg)^{1/2},~~
S'(\vec{f})(x):=\bigg(\iint_{\Gamma(x)}\big|\mathcal{G}_t(\vec{f})(z)\big|^2\frac{dzdt}{t^{n+1}}\bigg)^{1/2},
\end{equation*}
and
\begin{equation*}
g^{\ast\ast}_{\lambda}(\vec{f})(x):=\bigg(\iint_{\mathbb{R}^{n+1}_{+}}\Big(\frac{t}{t+|x-z|}\Big)^{\lambda n}
\big|\mathcal{G}_t(\vec{f})(z)\big|^2\frac{dzdt}{t^{n+1}}\bigg)^{1/2},\quad \lambda>1.
\end{equation*}
We also assume that $\mathcal{T}'_{g}$ can be extended to a bounded multilinear operator for some $1\leq q_1,q_2,\dots,q_m<\infty$, $0<q<\infty$ with $1/q=\sum_{i=1}^m 1/{q_i}$; that is,
\begin{equation*}
\mathcal{T}'_{g}:L^{q_1}(\mathbb R^n)\times L^{q_2}(\mathbb R^n)\times \cdots\times L^{q_m}(\mathbb R^n)\rightarrow L^q(\mathbb R^n),
\end{equation*}
where $\mathcal{T}'_{g}$ denotes any one of the multilinear Littlewood--Paley functions with non-convolution type kernels.
\end{defin}

\begin{rem}
\begin{enumerate}
  \item If the kernel $\mathcal{K}$ is of the form $\mathcal{K}(x-y_1,x-y_2,\dots,x-y_m)$, i.e., in the form of convolution type, then $\mathcal{T}'_{g}$ coincides with the operator defined in Section \ref{sec13}.
  \item For the theory on multilinear Littlewood--Paley operators with more general kernels, see \cite{c}, \cite{sixue} and \cite{xueqing} for more details.
\end{enumerate}
\end{rem}

\begin{thm}[\cite{he}]
Let $2\leq m\in \mathbb{N}$, $1\leq p_1,p_2,\dots,p_m<\infty$ and $0<p<\infty$ with
\begin{equation*}
\frac{\,1\,}{p}=\frac{1}{p_1}+\frac{1}{p_2}+\cdots+\frac{1}{p_m}.
\end{equation*}
Then the following statements hold:

$(i)$ If each $p_i>1$, $i=1,2,\dots,m$, then there is a constant $C>0$ independent of $\vec{f}$ such that
\begin{equation*}
\big\|g'(\vec{f})\big\|_{L^p}\leq C\prod_{i=1}^m\|f_i\|_{L^{p_i}},
\quad \big\|S'(\vec{f})\big\|_{L^p}\leq C\prod_{i=1}^m\|f_i\|_{L^{p_i}},
\end{equation*}
hold for all $\vec{f}=(f_1,f_2,\dots,f_m)\in L^{p_1}(\mathbb R^n)\times L^{p_2}(\mathbb R^n)\times \cdots\times L^{p_m}(\mathbb R^n)$.

$(ii)$ If at least one $p_i=1$, then there is a constant $C>0$ independent of $\vec{f}$ such that
\begin{equation*}
\big\|g'(\vec{f})\big\|_{L^{p,\infty}}\leq C\prod_{i=1}^m\|f_i\|_{L^{p_i}},
\quad \big\|S'(\vec{f})\big\|_{L^{p,\infty}}\leq C\prod_{i=1}^m\|f_i\|_{L^{p_i}},
\end{equation*}
hold for all $\vec{f}=(f_1,f_2,\dots,f_m)\in L^{p_1}(\mathbb R^n)\times L^{p_2}(\mathbb R^n)\times \cdots\times L^{p_m}(\mathbb R^n)$.
In particular, the multilinear operators $g'$ and $S'$ are bounded from $L^{1}(\mathbb R^n)\times L^{1}(\mathbb R^n)\times \cdots\times L^{1}(\mathbb R^n)$ into $L^{1/m,\infty}(\mathbb R^n)$.
\end{thm}

\begin{thm}[\cite{he}]\label{he}
Suppose that $\lambda>2m$ and $0<\gamma<\min\{{n(\lambda-2m)}/2,\delta\}$. Let $2\leq m\in \mathbb{N}$, $1\leq p_1,p_2,\dots,p_m<\infty$ and $0<p<\infty$ with
\begin{equation*}
\frac{\,1\,}{p}=\frac{1}{p_1}+\frac{1}{p_2}+\cdots+\frac{1}{p_m}.
\end{equation*}
Then the following statements hold:

$(i)$ If each $p_i>1$, $i=1,2,\dots,m$, then there is a constant $C>0$ independent of $\vec{f}$ such that
\begin{equation*}
\big\|g^{\ast\ast}_{\lambda}(\vec{f})\big\|_{L^p}\leq C\prod_{i=1}^m\|f_i\|_{L^{p_i}}
\end{equation*}
holds for all $\vec{f}=(f_1,f_2,\dots,f_m)\in L^{p_1}(\mathbb R^n)\times L^{p_2}(\mathbb R^n)\times \cdots\times L^{p_m}(\mathbb R^n)$.

$(ii)$ If at least one $p_i=1$, then there is a constant $C>0$ independent of $\vec{f}$ such that
\begin{equation*}
\big\|g^{\ast\ast}_{\lambda}(\vec{f})\big\|_{L^{p,\infty}}\leq C\prod_{i=1}^m\|f_i\|_{L^{p_i}}
\end{equation*}
holds for all $\vec{f}=(f_1,f_2,\dots,f_m)\in L^{p_1}(\mathbb R^n)\times L^{p_2}(\mathbb R^n)\times \cdots\times L^{p_m}(\mathbb R^n)$.
In particular, the multilinear operator $g^{\ast\ast}_{\lambda}$ is bounded from $L^{1}(\mathbb R^n)\times L^{1}(\mathbb R^n)\times \cdots\times L^{1}(\mathbb R^n)$ into $L^{1/m,\infty}(\mathbb R^n)$.
\end{thm}

By using similar arguments, we can see that all the BMO--BLO results derived above are also true for the multilinear operators $g'$, $S'$ and $g^{\ast\ast}_{\lambda}$. The details are omitted here.

\begin{thm}\label{mainthm21}
For any $\vec{f}=(f_1,f_2,\dots,f_m)\in [\mathrm{BMO}(\mathbb R^n)]^m$ and $2\leq m\in \mathbb{N}$, then $g'(\vec{f})$ is either infinite everywhere or finite almost everywhere, and in the latter case, we then have
\begin{equation*}
\big\|\big[g'(\vec{f})\big]^2\big\|_{\mathrm{BLO}}\lesssim\prod_{i=1}^m\big\|f_i\big\|^2_{\mathrm{BMO}},
\end{equation*}
and hence
\begin{equation*}
\big\|g'(\vec{f})\big\|_{\mathrm{BLO}}\lesssim\prod_{i=1}^m\big\|f_i\big\|_{\mathrm{BMO}}.
\end{equation*}
\end{thm}

\begin{thm}\label{mainthm22}
For any $\vec{f}=(f_1,f_2,\dots,f_m)\in [\mathrm{BMO}(\mathbb R^n)]^{m}$ and $2\leq m\in \mathbb{N}$, then $S'(\vec{f})$ is either infinite everywhere or finite almost everywhere, and in the latter case, we then have
\begin{equation*}
\big\|\big[S'(\vec{f})\big]^2\big\|_{\mathrm{BLO}}\lesssim\prod_{i=1}^m\big\|f_i\big\|^2_{\mathrm{BMO}},
\end{equation*}
and hence
\begin{equation*}
\big\|S'(\vec{f})\big\|_{\mathrm{BLO}}\lesssim\prod_{i=1}^m\big\|f_i\big\|_{\mathrm{BMO}}.
\end{equation*}
\end{thm}

\begin{thm}\label{mainthm23}
Suppose that $\lambda>3m+{(2\delta+2\gamma)}/n$ and $\delta,\gamma>0$. For any $\vec{f}=(f_1,f_2,\dots,f_m)\in[\mathrm{BMO}(\mathbb R^n)]^{m}$ and $2\leq m\in \mathbb{N}$, then $g^{\ast\ast}_{\lambda}(\vec{f})$ is either infinite everywhere or finite almost everywhere, and in the latter case, we have
\begin{equation*}
\big\|\big[g^{\ast\ast}_{\lambda}(\vec{f})\big]^2\big\|_{\mathrm{BLO}}\lesssim\prod_{i=1}^m\big\|f_i\big\|^2_{\mathrm{BMO}},
\end{equation*}
and hence
\begin{equation*}
\big\|g^{\ast\ast}_{\lambda}(\vec{f})\big\|_{\mathrm{BLO}}\lesssim\prod_{i=1}^m\big\|f_i\big\|_{\mathrm{BMO}}.
\end{equation*}
\end{thm}

Concerning the $L^{\infty}$--BLO estimates for multilinear Littlewood--Paley operators with non-convolution type kernels, we have the following
results. 
\begin{thm}
For any $\vec{f}=(f_1,f_2,\dots,f_m)\in [L^{\infty}(\mathbb R^n)]^{m}$ and $2\leq m\in \mathbb{N}$, then $g'(\vec{f})$ is finite everywhere,
\begin{equation*}
\big\|\big[g'(\vec{f})\big]^2\big\|_{\mathrm{BLO}}\lesssim\prod_{i=1}^m\big\|f_i\big\|^2_{L^{\infty}},
\end{equation*}
and hence
\begin{equation*}
\big\|g'(\vec{f})\big\|_{\mathrm{BLO}}\lesssim\prod_{i=1}^m\big\|f_i\big\|_{L^{\infty}}.
\end{equation*}
\end{thm}

\begin{thm}
For any $\vec{f}=(f_1,f_2,\dots,f_m)\in [L^{\infty}(\mathbb R^n)]^{m}$ and $2\leq m\in \mathbb{N}$, then $S'(\vec{f})$ is finite everywhere, 
\begin{equation*}
\big\|\big[S'(\vec{f})\big]^2\big\|_{\mathrm{BLO}}\lesssim\prod_{i=1}^m\big\|f_i\big\|^2_{L^{\infty}},
\end{equation*}
and hence
\begin{equation*}
\big\|S'(\vec{f})\big\|_{\mathrm{BLO}}\lesssim\prod_{i=1}^m\big\|f_i\big\|_{L^{\infty}}.
\end{equation*}
\end{thm}

\begin{thm}
Assume that $\lambda>3m+{(2\delta+2\gamma)}/n$ and $\gamma,\delta>0$. For any $\vec{f}=(f_1,f_2,\dots,f_m)\in [L^{\infty}(\mathbb R^n)]^{m}$ and $2\leq m\in \mathbb{N}$, then $g^{\ast\ast}_{\lambda}(\vec{f})$ is finite everywhere, 
\begin{equation*}
\big\|\big[g^{\ast\ast}_{\lambda}(\vec{f})\big]^2\big\|_{\mathrm{BLO}}\lesssim\prod_{i=1}^m\big\|f_i\big\|^2_{L^{\infty}},
\end{equation*}
and hence
\begin{equation*}
\big\|g^{\ast\ast}_{\lambda}(\vec{f})\big\|_{\mathrm{BLO}}\lesssim\prod_{i=1}^m\big\|f_i\big\|_{L^{\infty}}.
\end{equation*}
\end{thm}

\section*{Acknowledgment}
The authors were supported by a grant from Xinjiang University under the project``Real-Variable Theory of Function Spaces and Its Applications". This work was also supported by the Natural Science Foundation of China  (No.XJEDU2020Y002 and 2022D01C407).

\end{document}